\documentclass[11pt]{amsart}

\usepackage{tabularx, hyperref}
\usepackage{amssymb} \usepackage{amsfonts} \usepackage{amsmath}
\usepackage{amsthm} \usepackage{epsfig, subfig}
\usepackage{ amscd, amsxtra, latexsym}
\usepackage[all]{xy}
\usepackage{caption}
\usepackage{enumerate}
\usepackage{color}
\newtheorem{lemma}{Lemma}[section]
\newtheorem{thm}[lemma]{Theorem}
\newtheorem{prop}[lemma]{Proposition}
\newtheorem{cor}[lemma]{Corollary}

\theoremstyle{definition}
\newtheorem{defn}[lemma]{Definition}

\newtheorem{rem}[lemma]{Remark}
\newtheorem{conv}[lemma]{Convention}
\theoremstyle{definition}
\newtheorem*{claim}{Claim}

\newcommand{\g} {\ensuremath {\gamma}}

\newcommand{\N}{\ensuremath {\mathbb{N}}}
\newcommand{\R} {\ensuremath {\mathbb{R}}}

\newcommand{\G} {\ensuremath {\mathbb{G}}}

\newcommand{\F} {\ensuremath {\mathbb{F}}}

\newcommand{\st}{\ensuremath {\, ^* }}

\newcommand{\calG} {\ensuremath {\mathcal{G}}}

\newcommand{\calP} {\ensuremath {\mathcal{P}}}

\newcommand{\calH} {\ensuremath {\mathcal{H}}}
\newcommand{\calW} {\ensuremath {\mathcal{W}}}

\newcommand{\calI} {\ensuremath {\mathcal{I}}}
\newcommand{\calY} {\ensuremath {\mathcal{Y}}}
\newcommand{\calJ} {\ensuremath {\mathcal{J}}}
\newcommand{\calU} {\ensuremath {\mathcal{U}}}
\newcommand{\calX} {\ensuremath {\mathcal{X}}}
\newcommand{\calF} {\ensuremath {\mathcal{F}}}

\begin{document}

\thispagestyle{empty}
\title{Tree-graded asymptotic cones}
\author{Alessandro Sisto}
\address{Mathematical Institute, 24-29 St Giles,Oxford OX1 3LB, United Kingdom}
\email{sisto@maths.ox.ac.uk}
\date{}
\maketitle

\begin{abstract}
We study the bilipschitz equivalence type of tree-graded spaces, showing that asymptotic cones of relatively hyperbolic groups (resp. asymptotic cones of groups containing a cut-point) only depend on the bilipschitz equivalence types of the pieces in the standard (resp. minimal) tree-graded structure. In particular, the asymptotic cones of many relatively hyperbolic groups do not depend on the scaling factor. We also describe the asymptotic cones as above ``explicitly''. Part of these results were obtained independently and simultaneously by D. Osin and M. Sapir in~\cite{OS}.
\end{abstract}

\section*{Introduction}
\let\thefootnote\relax\footnotetext{2010 \emph{Mathematics Subject Classification}. Primary: 20F65 Secondary: 20F69, 03H05 \newline Keywords: tree-graded, asymptotic cone, relatively hyperbolic}
Relatively hyperbolic groups were first introduced in~\cite{Gr2} as a generalization of hyperbolic groups. Equivalent definitions and
further properties may be found in \cite{Bowditch:RelHyp}, \cite{Fa},
\cite{Dahmani:thesis}, \cite{Yaman:RelHyp}, \cite{Os1}, \cite{DS1}. They are modeled on the fundamental groups of finite volume complete manifolds of negative curvature (while hyperbolic groups are modeled on the fundamental groups of \emph{compact} negatively curved manifolds). Other examples of relatively hyperbolic groups include free products of groups, which are hyperbolic relative to the factors, and fundamental groups of non-geometric Haken manifolds with at least one hyperbolic component, which are hyperbolic relative to the fundamental groups of the maximal graph manifold components and the fundamental groups of the tori or Klein bottles not bounding any graph manifold component.
\par
Tree-graded spaces were introduced by Dru\c{t}u and Sapir in~\cite{DS1} to describe the large scale geometry of relatively hyperbolic groups. These spaces have a distinguished family of subsets called \emph{pieces} such that, very roughly, each path can be thought of as a concatenation of paths contained in some piece and paths in a real tree. The tree-graded structure on the large scale has been used in~\cite{DS1}, for example, to prove rigidity theorems and to study the automorphisms of certain relatively hyperbolic groups.
\par
The link between tree-graded spaces and relatively hyperbolic groups is provided by asymptotic cones. The asymptotic cones of a metric space $X$ are defined in order to keep track of the large scale geometry of $X$. In fact, they are obtained rescaling the metric on $X$ by an ``infinitesimal factor'', so that ``infinitely far away'' points become close.
\par
Tree-graded spaces describe the large scale geometry of relatively hyperbolic groups in the sense that the asymptotic cones of a relatively hyperbolic group can be endowed with a natural tree-graded structure.
\par
The idea behind asymptotic cones was introduced by Gromov to prove that groups of polynomial growth are virtually nilpotent in~\cite{Gr1}, where the notion of Gromov-Hausdorff convergence is used to define them. Later, a more general and somehow easier to handle definition was given by Van den Dries and Wilkie in~\cite{vDW}, using nonstandard methods. However, instead of using nonstandard methods, asymptotic cones are nowadays defined in terms of ultrafilters. Since~\cite{Gr1}, asymptotic cones have been used in several ways, for example to prove quasi-isometric rigidity results, see~\cite{KlL, KaL1,KaL2,BKMM,FLS}.
\par

We will use a nonstandard definition of asymptotic cones which is a restatement of the one based on ultrafilters. We will do that because the formalism is much lighter and, much more important, some elementary properties of nonstandard extensions turn out to be very powerful tools in studying asymptotic cones, as we will see. Indeed, most proofs involving the ultrafilter based definition of asymptotic cones contain the proof of some particular cases of these properties. Also, the nonstandard definition of asymptotic cones is closer to the idea of looking at a metric space from infinitely far away than the ultrafilter based one, which is closer to the original concept due to Gromov of convergence of rescaled metric spaces. This convergence is more difficult to ``visualize''.

\subsection*{Main results}

 We consider groups hyperbolic relative to a collection of proper subgroups and, more generally, groups with asymptotic cones having global cut-points. For the latter type of asymptotic cones we consider the \emph{minimal} tree-graded structure as defined in \cite{DS1} (i.e. a structure in which the pieces are all the connected subsets which consist of a single point or do not contain global cut-points). We refer to the cases mentioned above as the relatively hyperbolic case and the minimal case, respectively.
\par

We will provide a simple condition for relatively hyperbolic groups to have bilipschitz equivalent asymptotic cones (Theorem~\ref{comparable:thm}) and for asymptotic cones endowed with the minimal tree-graded structure to be bilipschitz equivalent (Theorem~\ref{cutpinthecone:thm}). The condition for relatively hyperbolic groups is also necessary in the case of groups hyperbolic relative to unconstricted subgroups (see Remark~\ref{uncostricted:rem}). A simplified version of the statement of Theorem~\ref{comparable:thm} is provided below. In what follows, the asymptotic cone of a group $G$ with scaling factor $\nu$ will be denoted by $C(G,\nu)$.

\begin{thm}\label{comparableintro:thm}
Let $G_0,G_1$ be groups hyperbolic relative to their proper subgroups $H_0, H_1$, respectively. Suppose that $C(H_0,\nu_0)$ is bilipschitz equivalent to $C(H_1,\nu_1)$ for some scaling factors $\nu_0,\nu_1$. Then $C(G_0,\nu_0)$ is bilipschitz equivalent to $C(G_1,\nu_1)$.
\end{thm}

D. Osin an M. Sapir proved this independently in~\cite{OS}.
\par
Note that the theorem implies, in particular, that if $G$ is hyperbolic relative to $H$ and the asymptotic cones of $H$ are all bilipschitz equivalent, then so are the asymptotic cones of $G$. This is the case when $H$ is, for example, virtually nilpotent. Therefore this consideration applies to the case of the fundamental groups of finite volume (complete, non-compact) manifolds of pinched negative curvature.
\par
We also obtain the following result, in the minimal structure case.
\begin{thm}\label{cutpinthecone:thm}
Suppose that $C(G_0),C(G_1)$ are asymptotic cones of the (non-virtually cyclic) groups $G_0, G_1$ and that they contain cut-points. Let $\calP_{i}$ be the minimal tree-graded structure on $C(G_i)\, ,\, i=1,2\, ,$ and suppose that for each $P\in\calP_j$ there exists $P'\in\calP_{j+1}$ such that $P$ and $P'$ are bilipschitz equivalent (with the same bilipschitz constant). Then $C(G_0)$ and $C(G_1)$ are bilipschitz equivalent.
\end{thm}
D. Osin an M. Sapir proved this in~\cite{OS} under the Continuum Hypothesis.
\par
The theorems above can be applied to study asymptotic cones of $3-$manifold groups \cite{Si3}.
\par
An important intermediate step in the proof of the theorems above is the following result of independent interest (see Propositions~\ref{valency:prop} and~\ref{transvmin:prop}).
\begin{thm}
In each asymptotic cone of a relatively hyperbolic group (resp. asymptotic cone a group containing a cut-point) endowed with the standard (resp. minimal) tree-graded structure, the valency of each transversal tree is $2^{\aleph_0}$.
\end{thm}

This result has been proven independently in~\cite{OS} in the relatively hyperbolic case and under the Continuum Hypothesis, which we do not assume, also in the minimal case.
\par
For convenience and simplicity, we will fix an ultrafilter and state our results for asymptotic cones constructed using it. However, the proofs give analogue results when we allow the ultrafilter to change as well. This is made more precise in Remark~\ref{ultrafilter:rem}.
\par
This paper is based, for the most part, on the author's Master's thesis~\cite{Th}, defended at the University of Pisa on June 28, 2010.

\subsection*{Acknowledgments}
The author would like to greatly thank his Master's thesis advisor Roberto Frigerio. He is also grateful to Yves de Cornulier for helpful comments.

\section{Nonstandard extensions}

For the following sections we will need basic results about the theory of nonstandard extensions. The treatment will be rather informal, for a more formal one and the proofs of the results in this section see for example~\cite{Go}.
\par
Roughly speaking, nonstandard extensions are given by a ``$\st$ map'' such that
\begin{enumerate}
\item
for each set $X$, $\st X$ is a set naturally containing $X$,
\item
for each function $f:X\to Y$, $\st f: \st X\to\st Y$ is a function which extends $f$ (this makes sense in view of the inclusion $X\subseteq \st X$,
\item
if $R$ is a $n-$ary relation on $X$, $\st R$ is an $n-$ary relation on $\st X$ extending $R$.
\end{enumerate}
We will start by stating the basic properties of nonstandard extensions, and we will provide a construction only in the end. This is to emphasize that the properties are more important than the actual construction, and most of the times they are all that is needed to know.
\par
Let us first state a result about the properties that the ``standard world'' and the ``nonstandard world'' have in common, then we will study the extra properties of the nonstandard extensions.

\begin{defn}
A formula $\phi$ is \emph{bounded} if all quantifiers appear in expressions like $\forall x\in X$, $\exists x \in X$ (bounded quantifiers).
\par
The \emph{nonstandard interpretation of} $\phi$, denoted $\st\phi$, is obtained by adding $\st$ before any set, relation or function (not before quantified variables).
\end{defn}

An example will make these concepts clear: consider
$$\forall X\subseteq \N, X\neq\emptyset\ \exists x\in X\ \forall y\in X\ x\leq y,$$
which expresses the fact that any non-empty subset of $\N$ has a minimum. This formula is not bounded, because it contains ``$\forall X\subseteq \N$''. However, it can be turned into a bounded formula by substituting ``$\forall X\subseteq \N$'' with ``$\forall X\in\calP(\N)$''. The nonstandard interpretation of the modified formula reads
$$\forall X\in\st\calP(\N), X\neq\st\emptyset\ \exists x\in X\ \forall y\in X\ x\st\leq y.\ \ \ \ \ (1)$$
The following theorem will also be referred to as the \emph{transfer principle}.
\begin{thm}\label{transf:thm}
(\L o\v{s} Theorem) Let $\phi$ be a bounded formula. Then $\phi\iff\st\phi$.
\end{thm}

Easy consequences of this theorem are, for example, that the nonstandard extension $(\st G,\st\cdot)$ of a group $(G,\cdot)$ is a group, or that the nonstandard extension $(\st X,\st d)$ of a metric space $(X,d)$ is a $\st\R-$metric space (that is $\st d:\st X\times\st X\to\st\R$ satisfies the axioms of distance, which make sense as $\st\R$ is in particular an ordered abelian group). To avoid too many $\st$'s, we will often drop them before functions or relations, for example we will denote the ``distance'' on $\st X$ as above simply by ``$d$'', the order on $\st\R$ by ``$\leq$'' and the group operation on $\st G$ by ``$\cdot$''.
In view of the transfer principle, the following definition will be very useful:

\begin{defn}
$A\subseteq\st X$ will be called \emph{internal} subset of $X$ if $A\in\st\calP(X)$. An internal set is an internal subset of some $\st X$.
\par
$f:\st X\to\st Y$ will be called \emph{internal} function if $f\in\st(Y^X)=\st\{f:X\to Y\}$.
\end{defn}

One may think that ``living inside the nonstandard world'' one only sees internal sets and functions, and therefore, by the transfer principle, one cannot distinguish the standard world from the nonstandard world.
\par

The following inclusions can be proven using the transfer principle:
$$\{\st A: A\in\calP(X)\}\subseteq \st \calP(X)\subseteq \calP(\st X),$$
$$\{\st f: f\in Y^X\}\subseteq \st (Y^X) \subseteq (\st Y)^{\st X}.$$
The equalities are in general very far from being true, as we will see.
\par
Another example: the transfer principle applied to $(1)$, which tells that each non-empty subset of $\N$ has a minimum, gives that each \emph{internal} non-empty subset of $\st X$ has a minimum ($\st\emptyset=\emptyset$ as, for each set $A$, $\exists a\in A\iff \exists a\in\st A$).
\par
Let us now introduce a convention we will often use. For each definition in the ``standard world'' there exists a nonstandard counterpart. For example, the definition of geodesic (in the metric space $X$), yields the definition of $\st$geodesics, in the following way. The definition of geodesic (with domain the interval $[0,1]$, for simplicity) can be given as
$$\gamma\in X^{[0,1]}{\rm \ is\ a\ geodesic} \iff \forall x,y\in[0,1]\: d(\gamma(x),\gamma(y))=|x-y|,$$
Therefore the definition of $\st$geodesic can be given as
$$\gamma\in \st(X^{[0,1]}){\rm \ is\ a\ }\st{\rm geodesic} \iff \forall x,y\in\st[0,1]\: d(\gamma(x),\gamma(y))=|x-y|,$$
\par
Nonstandard extensions enjoy another property, which will be referred to as \emph{saturation}. First, a definition, and then the statement.

\begin{defn}
A collection of sets $\{A_j\}_{j\in J}$ has the \emph{finite intersection property} (FIP) if for each $n\in\N$ and $j_0,\dots,j_n\in J$, we have $A_{j_0}\cap\dots\cap A_{j_n}\neq\emptyset$.
\end{defn}

\begin{thm}\label{sat:thm}
Suppose that the collection of internal sets $\{A_n\}_{n\in\N}$ has the FIP. Then $\bigcap_{n\in\N} A_n\neq \emptyset$.
\end{thm}

Let us use this theorem to prove that $\st\R$ contains infinitesimals. It is enough to consider the collection of sets $\{\st (0,1/n)\}_{n\in\N^+}$ and apply the theorem to it. Note that for $n\in\N^+$, $\st (0,1/n)\in\st\calP(\R)$ as it is of the form $\st A$ for $A\in\calP(\R)$.
More generally, however, for each $x,y\in\st\R$, $(x,y)\in\st\R$ (we should use a different notation for intervals in $\R$ and intervals in $\st\R$, but hopefully it will be clear from the context which kind of interval is under consideration). This can be proven using the transfer principle.
\par
Using this one can show the following.

\begin{lemma}\label{cof:lem}
Let $\{\xi_n\}_{n\in\N}$ be a sequence of infinitesimals. There exists an infinitesimal $\xi$ greater than any $\xi_n$.

\end{lemma}








Now, some lemmas which are frequently used when working with nonstandard extensions. The first one is usually referred to as \emph{overspill}:

\begin{lemma}
Suppose that the internal subset $A\subseteq\st\R^+$ (or $A\subseteq\st\N$) contains, for each $n\in\N$, an element greater than $n$. Then $A$ contains an infinite number.

\end{lemma}

\begin{lemma}\label{underspill:lem}
Suppose that the internal subset $A\subseteq\st\R^+$ is such that, for each $n\in\N^+$, $A\cap\{x\in\st\R:x<1/n\}\neq \emptyset$. Then $A$ contains an infinitesimal.

\end{lemma}



Let us introduce some (quite intuitive) notation, which we are going to use from now on.

\begin{defn}
Consider $\xi,\eta\in\st\R$, with $\eta\neq 0$. We will write:
\begin{itemize}
\item
$\xi\in o(\eta)$ (or $\xi\ll\eta$ if $\xi,\eta$ are nonnegative) if $\xi/\eta$ is infinitesimal,
\item
$\xi\in O(\eta)$ if $\xi/\eta$ is finite,
\item
$\xi\gg\eta$ if $\xi,\eta$ are nonnegative and $\xi/\eta$ is infinite,
\item
$\xi\equiv\eta$ if $\xi\in O(\eta)\backslash o(\eta)$.
\end{itemize}


\end{defn}

The map we give by the following lemma plays a fundamental role in nonstandard analysis, and will be used in the definition of asymptotic cone:

\begin{prop}
There exists a map $st:O(1)\to\R$ such that, for each $\xi\in\st\R$, $\xi-st(\xi)$ is infinitesimal.

\end{prop}

We will call $st(\xi)$ the standard part of $\xi$. Note that $st(\xi)=0\iff \xi$ is infinitesimal.

\subsection{Construction of nonstandard extensions}
In this section we will briefly describe the construction of nonstandard extensions which is (implicitly) used to define the asymptotic cones in virtually every article in which they are used.
\par
Fix a (non-principal) ultrafilter $\calU$ on $\N$. If $X$ is any set, consider the equivalence relation $\sim$ on $X^\N$ defined by
$$(x_n)_{n\in\N}\sim (y_n)_{n\in\N}\iff \{n\in N: x_n=y_n\}\in\calU$$
We can define, for each set $X$,
$$\st X= X^\N/_\sim.$$
The inclusion of $X$ into $\st X$ is given by constant sequences, that is $x\in X$ is mapped to the sequence with constant value $x$.
\par
Nonstandard extensions of functions and relations can be defined componentwise, for example if $f:X\to Y$ is any function we set $\st f([(x_n)_{n\in\N}])=[(f(x_n))_{n\in\N}]$.

\begin{conv}
The definition of the nonstandard extensions depends on the ultrafilter $\calU$. From now on we fix an ultrafilter $\calU$ on $\N$, and we will consider the nonstandard extensions constructed using $\calU$.
\end{conv}

The following remark gives the only property which we will need that depends on the particular construction of the nonstandard extensions we chose.

\begin{rem}
The nonstandard extension of a set of cardinality at most $2^{\aleph_0}$ has cardinality at most $2^{\aleph_0}$.
\end{rem}

Finally, a remark on how to translate the nonstandard language in the ultrafilter language, and vice versa.

\begin{rem}
Internal subsets (resp. functions) are ultralimits of subsets (resp. functions).
\end{rem}

\subsection{Asymptotic cones}

Let $(X,d)$ be a metric space. The asymptotic cones of $X$ are ``ways to look at $X$ from infinitely far away''. Let us make this idea precise.

\begin{defn}
Consider $\nu\in\st\R$, $\nu\gg 1$. Define on $\st X$ the equivalence relation $x\sim y\iff d(x,y)\in o(\nu)$.
The \emph{asymptotic cone $C(X,p,\nu)$ of $X$ with basepoint $p\in\st X$ and scaling factor $\nu$} is defined as
$$\{[x]\in\st X/\sim: d(x,p)\in O(\nu)\}.$$
The distance on $C(X,p,\nu)$ is defined as $d([x],[y])=st\left(\st d(x,y)/\nu\right)$.
\end{defn}

This definition of asymptotic cone is basically due to van den Dries and Wilkie, see~\cite{vDW}. However, the original concept is due to Gromov, see~\cite{Gr1}. The aim of~\cite{vDW} was to simplify the proofs in~\cite{Gr1}.
\par
Before proceeding, a few definitions. If $q\in\st X$ and $d(p,q)\in O(\nu)$, so that $[q]\in C(X,p,\nu)$, then $[q]$ will be called the \emph{projection of $q$ on }$C(X,p,\nu)$. Similarly, if $A\subseteq \{x\in\st X: d(x,p)\in O(\nu)\}$, the projection of $A$ on $C(X,p,\nu)$ is $\{[a]|a\in A\}$. If $A\subseteq \st X$ is not necessarily contained in $\{x\in\st X: d(x,p)\in O(\nu)\}$, we will call $\{[a]\in C(X,p,\nu)|a\in A\}$ the \emph{set induced by }$A$.
\par

Let us now introduce the asymptotic cones we are mostly interested in.

\begin{defn}
Let $G$ be a finitely generated group and $S$ a finite generating set for $G$. The \emph{asymptotic cone $C_S(G,g,\nu)$ of $G$ with basepoint $g\in\st G$ and scaling factor }$\nu\gg 1$ is $C(\mathcal{CG}_S(G),g,\nu)$, where $\mathcal{CG}_S(G)$ denotes the Cayley graph of $G$ with respect to $S$.

\end{defn}

Here are the basic properties of asymptotic cones of groups (see~\cite{Dr1}).

\begin{lemma}
For any finitely generated group $G$, finite generating sets $S,S'$, $g,g'\in\st\G$, $\nu\gg 1$:

\begin{itemize}
\item
$C_S(G,g,\nu)$ is complete, geodesic and homogeneous,
\item
$C_S(G,g,\nu)$ is isometric to $C_S(G,g',\nu)$,
\item
$C_S(G,g,\nu)$ is $k-$bilipschitz homeomorphic to $C_{S'}(G,g,\nu)$.
\end{itemize}

\end{lemma}

When a finite generating set $S$ is fixed, we will often write $C(G,g,\nu)$ instead of $C_S(G,g,\nu)$.
\par
We will often identify $G$ with the set of vertices of $\mathcal{CG}_S(G)$.

\section{Tree-graded spaces}
The asymptotic cones we will be interested in have the structure which we will describe in this section. All results and definitions before Lemma~\ref{concat:lem} are taken from~\cite{DS1}.

\begin{defn}
A geodesic complete metric space $\F$ is \emph{tree-graded} with respect to a collection $\calP$ of closed geodesic subsets of $\F$ (called \emph{pieces}) that cover $\F$ if the following properties are satisfied:
\par
$(T_1)$ two different pieces intersect in at most one point,
\par
$(T_2)$ each geodesic simple triangle is contained in one piece.
\end{defn}


\begin{conv}
Throughout the section, let $\F$ denote a tree-graded space with respect to $\calP$.
\end{conv}



\begin{lemma}
Each simple \emph{loop} in $\F$ is contained in one piece.
\end{lemma}

In particular, for example, simple quadrangles are contained in one piece.
\par
The most powerful technical tool for studying tree-graded spaces are the projections defined in the following lemma.

\begin{lemma}
For each $P\in\calP$ there exists a map $\pi_P:\F\to P$, called the projection on $P$, such that for each $x\in\F$:
\begin{itemize}
\item
$d(x,P)=d(x,\pi_P(x))$,
\item
each curve (in particular each geodesic) from $x$ to a point in $P$ contains $\pi_P(x)$,
\item
$\pi_P$ is locally constant outside $\F$. In particular, if $A\subseteq \F$ ($A\neq\emptyset$) is connected and $|A\cap P|\leq 1$, $\pi_P(A)$ consists of one point.
\end{itemize}
\end{lemma}

Note that if $x\in P$, then $\pi_P(x)=x$. The following results can be easily proven using projections.

\begin{cor}\label{emptypointsubgeod:cor}
Each arc (i.e. injective path) connecting 2 points of a piece $P$ is contained in $P$. In particular the intersection between a geodesic and a piece is either empty, a point or a subgeodesic.
\end{cor}



\begin{cor}\label{diffproj:cor}
If $x,y$ are such that $\pi_P(x)\neq\pi_P(y)$, for some piece $P$, then any geodesic $\delta$ from $x$ to $y$ intersects $P$.
\end{cor}

Another useful concept is that of transversal tree, defined below.

\begin{defn}\label{transv:defn}
For each $x\in\F$ denote by $T_x$ the set of points $y\in\F$ such that there exists a path joining $x$ to $y$ which intersects each piece in at most one point. Such sets will be called \emph{transversal trees}.
\end{defn}

Basic properties of \emph{transversal trees} are given below.

\begin{lemma}
For each $x\in\F$
\begin{itemize}
\item
$T_x$ is a real tree,
\item
$T_x$ is closed in $\F$,
\item
if $y\in T_x$, then $T_x=T_y$,
\item
every arc joining $y,x\in T_x$ is contained in $T_x$.
\end{itemize}
\end{lemma}

The next lemma will be used a lot of times.

If $\beta$ and $\gamma$ are geodesics such that the final point of $\beta$ is the initial point of $\gamma$, we will denote their concatenation by $\beta\gamma$.

\begin{lemma}\label{concat:lem}

Suppose that $\gamma_1$ and $\gamma_2$ are geodesics or geodesic rays in a tree-graded space such that
\begin{enumerate}
\item
the final point $p$ of $\gamma_1$ is the starting point of $\gamma_2$,
\item
$\gamma_1\cap\gamma_2=\{p\}$,
\item
there is no piece containing a final subpath of $\gamma_1$ and an initial subpath of $\gamma_2$.
\end{enumerate}
Then $\gamma_1\gamma_2$ is a geodesic (or a geodesic ray or a geodesic line). Also, each geodesic from a point in $\gamma_1$ to a point in $\gamma_2$ contains $p$.

\end{lemma}

\begin{proof}
If the conclusion were false, we would have points $q\in\gamma_1$, $r\in\gamma_2$ such that $d(q,r)<d(q,p)+d(p,r)$. Consider a geodesic triangle with vertices $p,q,r$ and $[q,p]\subseteq\gamma_1$, $[p,r]\subseteq\gamma_2$. Condition $(2)$ and $d(q,r)<d(q,p)+d(p,r)$ imply that it cannot be a tripod, for otherwise $[q,p]\cap[p,q]$ should contain a non-trivial geodesic. Therefore there exists a piece $P$ intersecting both $[q,p]$ and $[p,r]$. Condition $(3)$ implies that $P$ does not contain $p$. But then both $[q,p]$ and $[p,r]$ should pass through $\pi_P(p)\neq p$, which contradicts $(2)$.
\par
The last part of the statement has a similar proof.\end{proof}




\begin{defn}
We will say that $\gamma_1$ and $\gamma_2$ \emph{concatenate well} if they satisfy the hypothesis of Lemma~\ref{concat:lem}.
\end{defn}



\subsection{Alternative definition of tree-graded spaces}
In this subsection we give a characterization of tree-graded spaces that will turn out to be more effective in the proof that certain spaces are tree-graded.
\par
Throughout the subsection, let us denote by $X$ a complete geodesic metric space and by $\calP$ a collection of subsets of $X$. We want to capture the fundamental properties of projections on a piece in a tree-graded space.

\begin{defn}
A family of maps $\Pi=\{\pi_P:X\to P\}_{P\in \calP}$ will be called \emph{projection system for} $\calP$ if, for each $P\in\calP$,
\par
$(P1)$ for each $r\in P$, $z\in X$, $d(r,z)=d(r,\pi_P(z))+d(\pi_P(z),z)$,
\par
$(P2)$ $\pi_P$ is locally constant outside $P$,
\par
$(P3)$ for each $Q\in\calP$ with $P\neq Q$, we have that $\pi_P(Q)$ is a point.
\end{defn}

\begin{rem}
Note that $\pi_P(x)$ is a point which minimizes the distance from $x$ to $P$. In particular, $P$ is closed. Also, if $x\in P$, then $\pi_P(x)=x$.
\end{rem}

\begin{lemma}\label{propproj:lem}
Suppose that $\{\pi_P\}_{P\in \calP}$ is a projection system.
\begin{enumerate}
\item
Consider $x\in X$ and $P\in\calP$. Each arc (in particular, each geodesic) from $x$ to some $p\in P$ passes through $\pi_P(x)$.
\item
For each $P\in\calP$, each arc (in particular, each geodesic) connecting 2 points in $P$ is entirely contained in $P$. As a consequence, the intersection between an arc $\gamma$ and $P\in\calP$ is either empty, a point or a subarc of $\gamma$.
\item
each simple loop which intersects some $P\in\calP$ in more than one point is contained in $P$.
\end{enumerate}
\end{lemma}

\begin{proof}
$(1)$ Consider an arc $\gamma:[0,t]\to\F$ from $x$ to $p$. Let $q=\gamma(u)$ be the first point of $\gamma\cap P$ ($P$ is closed in $X$ by assumption). By $(P2)$, $\pi_P\circ\gamma|_{[0,u)}$ is constant, so $\pi_P(x)=\pi_P(\gamma(u'))$ for each $u'\in[0,u)$. Using this fact and $(P1)$ with $r=q$ and $z=\gamma(u')$, for $u'\in[0,u)$, we get
$$d(q,\gamma(u'))=d(q,\pi_P(x))+d(\pi_P(x),\gamma(u')).$$
As $u'$ tends to $u$, the left-hand side tends to $d(q,q)=0$, while the right-hand side tends to $2d(q,\pi_P(x))$. Therefore $d(q,\pi_P(x))=0$ and $q=\pi_P(x)$.
\par
$(2)$ Consider an arc $\gamma$ between two points in some $P\in\calP$ and suppose by contradiction that there exists $x\in(\gamma\backslash P)$. We can consider a subarc $\gamma'$ of $\gamma$ containing $x$ and with endpoints $x_1\neq x_2$ with the property that $\gamma'\cap P=\{x_1,x_2\}$. We have that $[x,x_1]$ intersects $P$ only in its endpoint. By what we proved so far, we must have $\pi_P(x)=x_1$. But, for the very same reason, we should also have $\pi_P(x)=x_2$, a contradiction.
\par
$(3)$ The loop as in the statement can be considered as the union of 2 arcs connecting points on $P$. The conclusion follows from point $(2)$.\end{proof}

The characterization of projection systems given below will be helpful for future arguments.

\begin{lemma}\label{alternproj:lem}
Properties $(P1)$ and $(P2)$ can be substituted by:
\par
$(P'1)$ for each $P\in\calP$ and $x\in P$, $\pi_P(x)=x$,
\par
$(P'2)$ for each $P\in\calP$ and for each $z_1,z_2\in X$ such that $\pi_P(z_1)\neq\pi_P(z_2)$,
$$d(z_1,z_2)=d(z_1,\pi_P(z_1))+d(\pi_P(z_1),\pi_P(z_2))+d(\pi_P(z_2),z_2).$$
\end{lemma}

\begin{proof}
Assume that $\{\pi_P\}$ satisfies $(P'1)$ and $(P'2)$. Property $(P1)$ is not trivial only if $r\neq \pi_P(z)$, and in this case follows from $(P'2)$ setting $z_1=z$, $z_2=r$ and taking into account that, by $(P'1)$, $\pi_P(r)=r\neq\pi_P(z)$. As we have property $(P1)$, we also have that $d(z,P)=d(z,\pi_P(z))$ for each $z\in X$. Hence, property $(P2)$ follows from the fact that if $\pi_P(z_1)\neq\pi_P(z_2)$ then $d(z_1,z_2)>d(z_1,P)$.
\par
Assume that $\{\pi_P\}$ satisfies $(P1)$ and $(P2)$. We already remarked that $(P'1)$ holds. Consider $z_1$, $z_2$, $P$ as in property $(P'2)$, and a geodesic $\delta$ between $z_1$ and $z_2$. If we had $\delta\cap P=\emptyset$, then $\pi_P$ would be constant along $\delta$ and so $\pi_P(z_1)=\pi_P(z_2)$. Therefore $\delta$ intersects $P$. So, by point $(1)$ of the previous lemma, $\delta$ contains $\pi_P(z_1)$ and $\pi_P(z_2)$, hence the thesis.
\end{proof}

\begin{defn}
A geodesic is $\calP-$\emph{transverse} if it intersects each $P\in\calP$ in at most one point. A geodesic triangle in $X$ is $\calP-$\emph{transverse} if each side is $\calP-$transverse.
\par
$\calP$ is \emph{transverse-free} if each $\calP-$transverse geodesic triangle is a tripod.
\end{defn}

\begin{thm}\label{projgrad:thm}
Let $X$ be a complete geodesic metric space and let $\calP$ a collection of subsets of $X$ that covers $X$. Then $X$ is tree-graded with respect to $\calP$ if and only if $\calP$ is transverse-free and there exists a projection system for $\calP$.
\end{thm}

\begin{proof}
$\Rightarrow:$ This implication follows from the properties of tree-graded spaces we stated before.
\par
$\Leftarrow:$ Let $\Pi=\{\pi_P\}$ be a projection system for $P$.
Let us prove property $(T_1)$. Consider $P, Q\in\calP$ with $P\neq Q$. If $x,y\in P\cap Q$, we have $\pi_P(Q)\supseteq \{x,y\}$. By $(P3)$, this implies $x=y$.
\par
Let us show how to obtain property $(T_2)$. Consider a simple geodesic triangle $\Delta$ with vertices $a,b,c$. If it consists of one point (recall that we consider these triangles to be simple), then it is contained in some $P\in\calP$, as we assume that elements of $\calP$ cover $X$. So, we can suppose that $\Delta$ is not trivial. Then, it cannot be $\calP-$transverse, for otherwise it would be a non-trivial tripod, and therefore not a simple triangle.
\par
So, we can assume that $P\cap [a,b]$ contains a non-trivial subgeodesic $[a',b']\subseteq [a,b]$, for some $P\in\calP$. The conclusion follows from Lemma~\ref{propproj:lem}$-(3)$, as $\Delta$ is in particular a simple loop.
\end{proof}

\begin{rem}\label{mixed:rem}
Property $(P3)$ was used only to prove $(T_1)$. Therefore, another way to prove that $X$ is tree-graded is to prove property $(T_1)$, properties $(P1)$ and $(P2)$ (or $(P'1)$ and $(P'2)$) for some family of maps $\{\pi_P\}$, and that $\calP$ is transverse-free.

\end{rem}

\begin{lemma}\label{onlygeod:lem}
Suppose that there exists a projection system for $\calP$. Consider points $p,q\in X$ such that there exists one $\calP-$transverse geodesic $\gamma$ from $p$ to $q$. Then $\gamma$ is the only geodesic from $p$ to $q$.
\end{lemma}

\begin{proof}
Consider a geodesic $\gamma'$ from $p$ to $q$. If $\gamma'$ is different from $\gamma$, then a simple loop obtained as the union of non-trivial subgeodesics of $\gamma, \gamma'$ is easily found. This loop is contained in some $P\in\calP$, so $\gamma$ cannot be $\calP-$transverse.
\end{proof}

\subsection{Relative hyperbolicity}

\begin{defn}
$X$ is \emph{asymptotically tree-graded with respect to }$\calP$ if each asymptotic cone $Y$ of $X$ is tree-graded with respect to the collection of the non-empty subsets of $Y$ induced by elements of $\st\calP$. Also, we require that, if two distinct elements of $\st\calP$ induce pieces of $Y$, these pieces intersect in at most one point.
\end{defn}

Asymptotically tree-graded spaces were first defined in~\cite{DS1}, see Definition 4.19. Note that the above definition is more easily stated, thanks to the nonstandard formalism. The following result is contained in \cite{DS1}

\begin{lemma}\label{M:lem}
Let $X$ be asymptotically tree-graded with respect to $\calP$. Then there exists $M$ with the following property. If $\gamma$ is a geodesic connecting $x$ to $y$, and $d(x,P),d(y,P)\leq d(x,y)/3$ for some $P\in\calP$, then $\gamma\cap N_M(P)\neq\emptyset$.

\end{lemma}

We are finally ready to define relatively hyperbolic groups. Let $G$ be a finitely generated group and let $H_1,\dots,H_n$ be finitely generated subgroups of $G$.

\begin{defn}[\cite{DS1}]
$G$ is \emph{hyperbolic relative to} $H_1,\dots,H_n$ if for some (hence every) finite generating system $S$ for $G$, $C_S(G)$ is asymptotically tree-graded with respect to $\{gH_i|g\in G,i=1,\dots,n\}$.
\end{defn}

\begin{conv}
To avoid trivial cases, we will always assume that each $H_i$ as above is infinite and has infinite index in $G$.
\end{conv}


\section{Homeomorphism type of tree-graded spaces}

To simplify the notation, when $G$ is a group we will denote by $C(G,\nu)$ the asymptotic cone $C(G,e,\nu)$, where $\nu$ is a scaling factor. The main results in this section are Theorem~\ref{cutpinthecone:thm} and the theorem we are about to state. First, a definition.

\begin{defn}
Consider groups $G_i$, for $i=0,1$, which are hyperbolic relative to the collections of proper subgroups $\calH_i=\{H^1_i,\dots,H^{n(i)}_i\}$, and $\nu_0,\nu_1\gg 1$. We will say that $G_0$ at scale $\nu_0$ is comparable with $G_1$ at scale $\nu_1$ if for each $H\in \calH_i$ such that $C(H,\nu_i)$ is not a real tree there exists $H'\in\calH_{i+1}$ such that $C(H',\nu_{i+1})$ is bilipschitz equivalent to $C(H,\nu_{i})$.
\end{defn}

\begin{rem}\label{uncostricted:rem}
If each $C(H^j_i,\nu_i)$ does not have cut-points and $C(G_0,\nu_0)$ is bilipschitz equivalent to $C(G_1,\nu_1)$, then $G_0$ at scale $\nu_0$ is comparable to $G_1$ at scale $\nu_1$, because a homeomorphism between tree-graded spaces whose pieces do not have cut-points preserves the pieces (see~\cite{DS1}).

\end{rem}

For example, suppose that $G$ is hyperbolic relative to subgroups whose (bilipschitz type of the) asymptotic cones do not depend on the scaling factor (this is the case if the subgroups are virtually nilpotent, for example). Then for each $\nu_0,\nu_1\gg 1$ we have that $G$ at scale $\nu_0$ is comparable with itself at scale $\nu_1$.
\par
The theorem is the following.

\begin{thm}\label{comparable:thm}
Suppose that $G_0$ and $G_1$ are relatively hyperbolic groups and that $G_0$ at scale $\nu_0$ is comparable with $G_1$ at scale $\nu_1$.
Then $C(G_0, \nu_0)$ is bilipschitz homeomorphic to $C(G_1,\nu_1)$.
\end{thm}

\begin{cor}
Suppose that $G$ is hyperbolic relative to subgroups whose (bilipschitz type of the) asymptotic cones do not depend on the scaling factor. Then the asymptotic cones of $G$ are all bilipschitz equivalent.

\end{cor}

\subsection{Hyperbolic elements and transversal trees}

Throughout the subsection $G$ will denote a group which is hyperbolic relative to its subgroups $H_1,\dots,H_n$. Recall that we always assume that each $H_i$ has infinite index in $G$ and is infinite. We also fix a finite system of generators $S$.
\par
In this subsection recall some algebraic properties of relatively hyperbolic groups discovered by Osin, and we apply them to determine the structure of transversal trees in their asymptotic cones.

\begin{defn}
A hyperbolic element of $G$ is an infinite order element which is not conjugated to any element of $H_i$, $i=1,\dots n$.
\end{defn}

\begin{lemma}{\cite[Corollary 4.5]{Os2}}
There exists a hyperbolic element $g\in G$.
\end{lemma}

Fix such $g$.

\begin{lemma}{\cite[Corollary 1.7]{Os2}}
$g$ is contained in a virtually cyclic subgroup $E(g)$ of $g$ such that $G$ is hyperbolic relative to $H_1,\dots,H_n, E(g)$.
\end{lemma}

We are ready to study transversal trees.
\par
Let $X$ be the asymptotic cone of a group $G'$ hyperbolic relative to $H'_1,\dots, H'_k$ with basepoint $e\in\st G$ and scaling factor $\nu$. We have that $X$ is asymptotically tree-graded with respect to the set of pieces $\calP=\calP_1\cup\dots\cup\calP_k$, where elements of $\calP_i$ are induced by left $\st$cosets of $\st H'_i$.
\par
Let us start with counting how many pieces contain a fixed point.

\begin{lemma}\label{cardpieces:lem}
For each $i\in\{1,\dots,k\}$ and $x\in X$, $P(i,x)=\{P\in\calP_i|x\in P\}$ has cardinality $2^{\aleph_0}$.

\end{lemma}

\begin{proof}
As $X$ is homogeneous through isometries which preserve the pieces (those induced by left translations by suitable elements of $\st G$), it is enough to determine the cardinality of $P(i,e)$. Consider the function $f:\N\to\N$ such that $f(n)$ is the number of left cosets of $H'_i$ which have a representative closer than $n$ to $e$. We have that $f$ is of course increasing an unbounded. In particular, for each infinite $\xi\in\st \N$, $f(\xi)$ is an infinite number.
Let us fix an infinite $\xi\in o(\nu)$. The left $\st$cosets counted by $f(\xi)$ give distinct elements of $P(i,e)$, so $|P(i,e)|\geq 2^{\aleph_0}$. Also $|X\backslash\{e\}|\leq 2^{\aleph_0}$ and, as different pieces can intersect in at most one point and each piece contains infinite points, $|P(i,e)|\leq |X\backslash\{e\}|$ (for each $P\in P(i,e)$ consider a point in $P$ different from $e$). So, we obtained the inequality $|P(i,e)|\leq 2^{\aleph_0}$, and hence the thesis.
\end{proof}

Now, let us focus on transversal trees in an asymptotic cone $Y$ of $G$ (recall that we denote the transversal tree at $e$ by $T_e$, see Definition \ref{transv:defn}). Note that they are isomorphic homogeneous trees, so we only need to study the valency of $T_e$ in $e$.

\begin{prop}\label{valency:prop}
The valency of $T_e$ in $e$ is $2^{\aleph_0}$.
\end{prop}

\begin{proof}
We have that $\st E(g)$ induces a line in $Y$, as it is quasi-isometrically embedded in $G$. This line intersects each piece induced by a left $\st$coset of some $H_i$ in at most one point, because this line belongs to a set of pieces including the sets induced by left $\st$cosets of the $H_i$'s, so property $(T1)$ applies. Using the previous lemma, we get that $T_e$ contains $2^{\aleph_0}$ geodesic lines. As the valency of $T_e$ cannot be more than $|T_e|\leq|X|\leq 2^{\aleph_0}$, it must be exactly $2^{\aleph_0}$.
\end{proof}

\subsection{Transversal trees in minimal tree-graded structures}
Let $G$ be a group and set $X=C(G,e,\nu)$. Also, to avoid trivialities, assume that $G$ is non-virtually cyclic. Suppose that $X$ contains a cut-point. Then, by the proof of~\cite[Lemma 2.31]{DS1}, we have that $X$ is tree-graded with respect to $\calP$, the collection of maximal subsets of $X$ which (consist of a single point or) do not contain cut-points. We will refer to $\calP$ as the \emph{minimal} tree-graded structure of $X$. Notice that for each $g\in\st G$ and $P\in\calP$ such that $gP$ is defined, we have $gP\in\calP$ (as the characterization of the elements of $\calP$ is invariant under isometries).

\begin{lemma}\label{manypieces:lem}
For each $P\in\calP$ which contains $e$ and does not consist of a single point there exists $g\in\st G$ such that $gP\neq P$ and $gP$ contains $e$.

\end{lemma}

\begin{proof}
Consider an element $g\in\st G$ such that $gP$ is defined and $[g]\notin P$. We claim that for each $p\in P$ we have $d(gp,p)\geq 2d(e,p)-d(e,[g])$. To show this, notice that $gP\neq P$ so that, for $x=\pi_P(gP),y=\pi_{gP}(P)$, we have $d(p,gp)=d(p,x)+d(x,y)+d(y,gp)$. This holds in particular for $p=e$, so
$$d(p,gp)+d(e,[g])\geq \left(d(p,x)+d(x,e)\right)+\left(d(y,gp)+d(ge,y)\right)$$
$$\geq d(e,p)+d(ge,gp)=2d(e,p),$$
what we wanted.
\par
Notice that for each $n\in\N^+$ we can find $g$ as above such that $d(e,[g])<1/n$.
Also, we have the following property, for each $n\in \N^+$ and some fixed $p\in P\backslash\{e\}$ and $q\in\st G$ with $[q]=p$:
\par
``each path in $\st G$ obtained concatenating at most $n$ internal geodesics connecting $q$ to $g q$ contains a point whose distance from $e$ is a most $\nu/n$.''
\par
Saturation (see also Lemma~\ref{underspill:lem}) gives that there exists $g$ satisfying $d(e,[g])<1/n$ and the property above for each $n\in\N^+$. In particular, $d(e,g)\in o(\nu)$, $d(gp,p)=2d(e,p)$ and all paths from $p$ to $gp$ contain $e$ (as paths can be approximated arbitrarily well by concatenations of geodesics induced by internal geodesics). Clearly, $g$ is as required.
\end{proof}

\begin{prop}\label{transvmin:prop}
 $T_e$ is an homogeneous tree with valency $2^{\aleph_0}$.
\end{prop}

\begin{proof}
 Notice that in the case that each $P\in\calP$ is a point, then $X$ is a real tree. By~\cite{Si}, $X$ is a point, a line or a homogeneous tree of valency $2^{\aleph_0}$. In the first 2 cases, $G$ is virtually cyclic (see~\cite[Proposition 6.1]{DS1}). So, the case when each $P\in\calP$ consists of a single point is set. Notice that (by homogeneity of $X$ and the definition of $\calP$) if there exists $P\in\calP$ containing at least 2 points, then the same is true for each element of $\calP$.
\par\smallskip
 \noindent {\bf Finding a ray contained in $T_e$.}
If $\gamma$ is a geodesic in $X$ (resp. an internal geodesic in $\st G$) any map $\pi:X\to\gamma$ (resp. $\pi:\st G\to\gamma$) satisfying $d(x,\gamma)=d(x,\pi(x))$ for each $x\in X$ (resp. for each $x\in\st G$) will be called a \emph{closest point projection} on $\gamma$.
\par
The idea is to use the fact that a transversal ray $\gamma$ has the property that the closest point projection on $\gamma$ satisfies property $(P2)$ (i.e. it is locally constant outside $\gamma$). This is easily seen as, for example, $\calP\cup\{\gamma\}$ gives a tree-graded structure for $X$ when $\gamma$ is a ray contained in $T_e$, and the closest point projection on a piece satisfies $(P2)$. Also, we claim that if a ray $\delta$ is not contained in a transversal tree then there is no closest point projection on $\delta$ satisfying $(P2)$.
In fact, by definition, $\delta$ intersects some $P\in\calP$ in a non-trivial subpath $\delta'$. As $P$ does not have cut-points, there exists a path $\alpha$ connecting the endpoints $x,y$ of $\delta'$ and not containing the midpoint $p$ of $\delta'$. However, if a closest point projection $\pi$ existed, it would be continuous at $x$ and $y$ and we would have $\pi(x)=x, \pi(y)=y$. This is easily seen to contradict $(P2)$.
\par
Now, we wish to show that a transversal ray $\g$ exists by showing that, informally, there are geodesics such that any closest point projection on them is arbitrarily close to satisfying $(P2)$, and then a saturation argument will quickly lead to the conclusion.
\par
Fix any $n\in\N^+$. If $\g$ is an internal geodesic, denote by $\phi(\gamma,n)$ the following property:
\par
``for each closest point projection $\pi$ on $\g$ and for each $x,y\in \st G$ such that $d(x,e),d(y,e)\leq n \nu$ and $d(x,y)\leq d(x,\gamma)/2$ we have $d(\pi(x),\pi(y))\leq \nu/n$.''
\par
 We will construct an internal geodesic $\g_n$ in $\st G$ satisfying the following properties:
\begin{itemize}
 \item $e$ is the starting point of $\g_n$,
 \item $l(\gamma_n)\geq n \nu$,
 \item $\gamma_n$ satisfies $\phi(\gamma_n,n)$.
\end{itemize}
Once we can construct such geodesics, we are done as saturation implies that there exists a geodesic $\gamma$ satisfying the above properties for each $n\in \N^+$. It is clear that $\gamma$ induces a transversal ray in $X$, by the previous discussion on properties of the closest point projection on (non-)transversal rays.
\par\smallskip
 \noindent {\bf Constructing the ``approximate transversal rays''.}
The idea to construct such geodesics is just to concatenate short geodesics in different pieces. Let us show how to construct $\gamma=\gamma_n$. We will first construct a geodesic $\delta$ in $X$. Choose a geodesic $\delta_0$ in $X$ starting at $e$ contained in a piece $P_0\in\calP$ and such that $0<l(\delta_0)\leq 1/(10n)$. Lemma~\ref{manypieces:lem} readily implies that there exists $g_0\in\st G$ such that $g_0 P_0\neq P_0$ and $[g_0]=p_0$. We can define inductively $\delta_i$ as the concatenation of $\delta_{j-1}$ and $g_0^{j}\delta_0$. Set $P_i=g_0^i P_0$. Notice that for each $i$ we have that $Q_i=\bigcup_{j\leq i} P_j$ is a piece in some tree-graded structure of $X$.
\par
Let $\gamma\subseteq \st G$ be an internal geodesic connecting $e$ to a representative of the endpoint $p_i$ of $\delta_i$, for $i$ big enough, and let $\delta$ be the induced geodesic in $X$. As $g_0^{j}\delta_0$ and $g_0^{j+1}\delta_0$ concatenate well, it is easily seen for $i$ big enough $\gamma$ has length at least $n\nu$.
Let $\pi$ be a closest point projection on $\gamma$. Notice that, for each $p\in X$, if $\pi_{Q_i}(p)\in P_j$, then $[\pi(q)]\in P_j$ as well, for any $q$ such that $[q]=p$.
Suppose that $d(\pi(x),\pi(y))\geq 10l(\delta_0)$, for some $x,y$ with $d(x,e),d(y,e)\in O(\nu)$. It is easily seen that for each $j,k$ such that $\pi_{Q_i}([x])\in P_j,\pi_{Q_i}([y])\in P_k$ we have $|j-k|\geq 5$. Fix such $j,k$ and suppose $j<k$. It is easy to show inductively that $\pi_{P_j}(P_k)=\{g_0^{k}\}$ and $\pi_{P_k}(P_j)=\{g_0^{j+1}\}$. So, we have that each geodesic from $[x]$ to $[y]$ contains $g_0^{k}$ and $g_0^{j+1}$, as it contains $\pi_{Q_i}([x])$ and $\pi_{Q_i}([y])$ and any geodesic connecting them contains the claimed points.
\par
This easily implies $d([x],[y])> d([x],[\pi(x)])$. In particular, given $x,y$ such that $d(x,y)\leq d(x,\gamma)/2$ then we must have $d(\pi(x),\pi(y))< 10 l(\delta_0)\nu\leq \nu/n$, and this shows $\phi(\gamma,n)$. To sum up, we showed that there exists $\gamma$ satisfying the second and third condition required for $\gamma_n$, and by definition of $\gamma$ the third condition is satisfied as well.
\par\smallskip
 \noindent {\bf Finding many rays contained in $T_e$.}
Up to now we proved that $T_e$ contains a ray $\gamma$. Next, let us use this ray to construct several other rays in $T_e$ (containing $e$). Consider a piece $P$ containing $e$. Consider an element $g\in\st G$ such that $[g]\neq e$ and $[g]\in P$. We have that $d(gp,p)=2d(e,p)+d(e,[g])$ for each $p\in\gamma$. Also, if $[g_1]\neq [g_2]$ and $[g_1],[g_2]\in P$, we have $d(g_1p,g_2p)=2d(e,p)+d([g_1],[g_2])$. Fix $q$ such that $[q]\in\gamma$ and $[q]\neq e$. In view of the considerations above, for each $n\in \N^+$ we can find $g_1,\dots,g_n$ such that
\begin{enumerate}
 \item $d(g_i q,g_j q)\geq 2 d(e,q)$ for $i\neq j$,
 \item $d(g_i, e)\leq 1/n$ for each $i$.
\end{enumerate}
By overspill, we can find an infinite $\mu\in\st\N$ and $g_1,\dots, g_\mu$ with the same properties for some infinite $\mu$. In particular, it is easily seen that the valency of $T_e$ at $e$ is at least $2^{\aleph_0}$ as $g_i[q]$ is not in the same connected component of $T_e\backslash\{e\}$ as $g_j[q]$ when $i\neq j$.
\par
The homogeneity of $T_e$ follows from the homogeneity of $X$ together with the fact that the set of pieces is invariant under isometries (and therefore the set of transversal trees is also invariant).
\end{proof}

\subsection{Geodesics in tree-graded spaces}
We are going to need some results about the structure of geodesics in tree-graded spaces. Throughout the subsection $\F$ will denote a tree-graded space with respect to the collection of proper subsets $\calP$. Unfortunately, it is not true that all geodesics in $\F$ are obtained by concatenation of geodesics in transversal trees or pieces, as shown by the ``fractal'' geodesics used in the proof of Lemma 6.11 in~\cite{DS1}. We want to analyze how far this is from being true.
\par

\begin{rem}\label{addingtrantree:rem}
If $\F$ is tree-graded with respect to $\calP$, then it is tree-graded also with respect to the collection of subsets $\calP'$ obtained from $\calP$ by adding a collection of disjoint transversal trees which cover $\F$. When $\F$ is considered as a tree-graded space with respect to $\calP'$, all its transversal trees are trivial.

\end{rem}

The above remark tells us how we can reduce to studying tree-graded spaces with trivial transversal trees. Henceforth, let $\F$ be such a tree-graded space.

\begin{defn}
Let $\gamma:[0,l]\to\F$ be a geodesic.
\begin{itemize}
\item
A piece interval is an interval $[a,b)\subseteq [0,l]$ (with $a<b$) such that $\gamma([a,b))$ is contained in a piece and $[a,b]$ is a maximal interval with this property.
\item
The piece subset $P_\gamma$ is the union of all piece intervals.
\end{itemize}

\end{defn}

\begin{rem}
A maximal interval $I$ such that $\gamma(I)$ is contained in a certain piece is closed because pieces are closed in $\F$.
\end{rem}

\begin{rem}
By the fact that different pieces intersect in at most one point, different piece intervals are disjoint.
\end{rem}

It is not true that, for each geodesic $\gamma:[0,l]\to\F$, $P_\gamma$ is the entire $[0,l)$, however:

\begin{lemma}\label{piecesubdense:lem}
$\overline{P_\gamma}=[0,l]$.

\end{lemma}

\begin{proof}
We have that if $x\in [0,l]\backslash\overline{P_\gamma}$ then $x$ is contained in some open interval $I$ such that no non-trivial interval $I'\subseteq I$ has the property that $\gamma(I')$ is contained in just one piece. We have that $\gamma (I)$ is contained in a transversal tree (by Corollary~\ref{emptypointsubgeod:cor}), a contradiction since transversal trees are trivial.
\end{proof}

The following two definitions are given in order to capture the properties of a geodesic in a tree-graded space with trivial transversal trees.
For short, we will call closed-open interval an interval closed on the left and open on the right.

\begin{defn}
An almost filling of an interval $[l,m]$ is a collection $\{I_a\}_{a\in A}$ of non trivial closed-open intervals in $[l,m]$ (in particular $A$ is at most countable) such that
\begin{enumerate}
\item
if $a\neq a'$, $I_a$ and $I_{a'}$ are disjoint,
\item
$\bigcup_{a\in A} I_a$ is dense in $[l,m]$.
\end{enumerate}

\end{defn}

Before giving the next definition, let us describe the idea behind it. A P-geodesic is something which wants to keep track of the following data:
\begin{itemize}
\item
the kind of pieces a certain geodesic $\gamma$ intersects non-trivially,
\item
the maximal intervals of the domain of $\gamma$ mapped in a piece (the $\overline{I_a}$'s, for $I_a$ as above),
\item
the last point on $\gamma\cap P$ for some $P$ which $\gamma$ intersects non-trivially ($\Gamma(t)$ for every $t$ varying in the appropriate interval $I_a$).
\end{itemize}
More precisely, it is the associated almost filling that keeps track of the first and second kind of information.

\begin{defn}
Suppose we are given a family of pointed metric spaces $\{(P_i,r_i)\}_{i\in I}$. A P-geodesic $\Gamma$ with associated almost filling $\{I_a\}_{a\in A}$ of an interval $[l,m]$ and range $\{(P_i,r_i)\}_{i\in I}$ is a function $\Gamma:\bigcup I_a\to \bigsqcup P_i$ such that
\begin{enumerate}
\item
$\Gamma|_{I_a}$ is constant for each $a\in A$,
\item
denoting by $h_\Gamma:\bigcup I_a\to I$ the function such that $\Gamma(t)\in P_i\iff h_\Gamma(t)=i$, we have $d(r_{h_\Gamma(t)},\Gamma(t))=l(I_a)$.

\end{enumerate}
The function $h_\Gamma$ will be called the \emph{index selector} for $\Gamma$.
\end{defn}

We could equivalently define $\Gamma$ as a function with domain $A$. The reason we chose this definition is merely technical.
\par
Suppose now that $\F$ is a homogeneous tree-graded space such that each piece is homogeneous (we still assume that transversal trees are trivial). Let $\{P_i\}$ be a choice of representatives of isometry classes of the pieces. For each $i$, fix a basepoint $r_i\in P_i$ and set $\calP=\{(P_i,r_i\}$.
\par

\begin{defn}
A choice of charts is the choice, for each piece $P$ and $p\in P$, of an isometry between $P$ and $P_i$ sending $p$ to $r_i$, for some $(P_i,r_i)\in\calP$.

\end{defn}

Suppose that for each pair $(x,P)$, where $P$ is a piece and $x$ is a point contained in $P$, we have a fixed an isometry between $P$ and some $P_i$ sending $x$ to the corresponding $r_i$. Finally, fix a basepoint $p\in \F$. Given this data, we can associate to each geodesic $\gamma$ in $\F$ parametrized by arc length a P-geodesic.

\begin{lemma}\label{Pgeod:lem}
If $\gamma:[0,l]\to\F$ is a geodesic in $\F$ parametrized by arc length, then:
\begin{enumerate}
\item
The collection $\calI_\gamma=\{I_a=[q_a,q'_a)\}_{a\in A_\gamma}$ of all maximal closed-open subintervals $J$ of $[0,l]$ such that $\gamma|_J$ is contained in one piece is an almost filling of $[0,l]$.
\item
Consider the function $h_\Gamma:\bigcup I_a\to I$ which associates to each $t$ the only $i\in I$ such that $\gamma|_{I_a}$ is contained in a piece isometric to $P_{h(t)}$, where $t\in I_a$. Also, let $\Gamma:\bigcup I_a\to \bigsqcup P_i$ be such that $\Gamma(t)$ is the point identified with $\gamma(q'_a)$ under the identification of $(\gamma(q_a),P)$ with $(p_{h_\Gamma(a)},P_{h_\Gamma(a)})$, where $P$ is the piece which contains $\gamma|_{I_a}$. Then $\Gamma$ is a P-geodesic and $h_\Gamma$ is its index selector.
\item
$\{I_a\}$ and $\Gamma$ depend only on the endpoints of $\gamma$.

\end{enumerate}

\end{lemma}

\begin{defn}
$\Gamma$ as above will be called the P-geodesic induced by $\gamma$.

\end{defn}

\begin{proof}
Lemma~\ref{piecesubdense:lem} implies $(1)$, and $(2)$ is clear.
\par
In order to prove $(3)$, we will prove that if $\gamma,\gamma'$ are geodesics from $p$ to $q$ and $\gamma$ intersect the piece $P$ in a non trivial arc, entering it in $x$ and leaving from $y$, then $\gamma'$ enters $P$ in $x$ and leave it from $y$ as well.
\par
First of all, we have to prove that $\gamma'$ intersects $P$. If this is not the case, then $(\gamma\cup \gamma')\backslash P$ is connected. But the projection of $\gamma\backslash P$ on $P$ consists of 2 points, and the projection of $\gamma'$ on $P$ consists of one point, as $\gamma'\cap P=\emptyset$. Therefore the projection of $(\gamma\cup\gamma')\backslash P$ on $P$ is not connected, a contradiction.
\par
Suppose now that $\gamma'$ enters $P$ in $x'\neq x$. Let $\overline{\gamma}$ (resp. $\overline{\gamma'}$) be the initial subgeodesic of $\gamma$ (resp. $\gamma'$) whose final point is $x$ (resp. $x'$). The projection of $\overline{\gamma}$ on $P$ is $x$ and the projection of $\gamma'$ on $P$ is $x'$. But $\overline{\gamma}\cap\overline{\gamma'}$ contains $p$, and therefore their projections on $P$ cannot be disjoint, a contradiction. One can proceed similarly for $y,y'$, considering final subgeodesics instead of initial subgeodesics.
\end{proof}

From now until the end of the subsection, fix a family $\{(P_i,r_i)\}_{i\in I}$ of homogeneous geodesic complete pointed metric spaces. Throughout the subsection all P-geodesics are implied to have range $\{(P_i,r_i)\}$.
\par
If $\calI$ is a family of subintervals of $[0,l]$ we set, for $x>0$, $\calI [x]=\{J\in\calI|J\subseteq [0,x]  \}$.

\begin{defn}
We will say that the P-geodesics $\Gamma$ and $\Gamma'$ with associated almost fillings, respectively, $\calI_\Gamma$ and $\calI_\Gamma'$ have the same P-pattern until $x>0$ if

\begin{enumerate}
\item
$\calI_\Gamma[x]=\calI_{\Gamma'}[x]$,
\item
$\Gamma(I)=\Gamma'(I)\ \forall I\in\calI_\Gamma[x]$,
\item
if there exists $J\in\calI_\Gamma$ such that $x\in J$ and $x$ is not the first point of $J$, then there exists $J'\in\calI_{\Gamma'}$ with the same property and $h_\Gamma(J)=h_{\Gamma'}(J')$.

\end{enumerate}

We will say that $\Gamma$ and $\Gamma'$ have the same initial P-pattern if there exists some $x>0$ such that $\Gamma$ and $\Gamma'$ have the same pattern until $x$.

\end{defn}

Clearly, having the same initial P-pattern is an equivalence relation on the set of P-geodesics. Denote by $\calW$ the quotient set.
\par
The property of having the same initial P-pattern is modeled on the following property for geodesics.

\begin{defn}
Let $\gamma$, $\gamma'$ be geodesics in $\F$ parametrized by arc length both starting from the same point $p$. We will say that $\gamma$ and $\gamma'$ have the same initial pattern if there exists $x>0$ and a piece $P$ such that $\gamma(x)$ and $\gamma'(x)$ both belong to $P$.

\end{defn}

\begin{lemma}
\begin{enumerate}
\item
Consider geodesics $\gamma$ and $\gamma'$ parametrized by arc length starting from $p$. If there exists a piece $P$ such that, for some $x>0$, $\gamma(x)$ and $\gamma'(x)$ both belong to $P$, then for each $0\leq y\leq x$ there exists a piece $P_y$ such that $\gamma(y),\gamma'(y)\in P_y$.
\item
Having the same initial pattern is an equivalence relation.
\end{enumerate}

\end{lemma}

\begin{proof}
$(1)$ If $p\in P$, the claim follows from the fact that each piece is convex. If this is not the case there exists $x'$ such that $\gamma(x')=\gamma'(x')=\pi_P(p)$, and, for $x'\leq y\leq x$, $\gamma(y),\gamma'(y)\in P$. For $0\leq y\leq x'$, and $y$ contained in a non-trivial interval $I$ such that $\gamma(I)\subseteq P'$ for some piece $P'$, the claim follows from the proof of Lemma~\ref{Pgeod:lem}, point $(3)$, which shows that $\gamma'(I)\subseteq P'$ as well. Also, if $I=[t_1,t_2]$ is maximal with that property, $\gamma(t_i)=\gamma'(t_i)$. If $y$ is not contained in such an interval, then $\gamma(y)=\gamma'(y)$ because the union of maximal intervals as above is dense in $[0,x']$, and so $y$ is the limit of a sequence of endpoints of such intervals.
\par
$(2)$ Consider geodesics parametrized by arc length $\gamma,\gamma',\gamma''$ and $x,y>0$ such that $\gamma(x)$ and $\gamma'(x)$ (resp. $\gamma'(y)$ and $\gamma''(y)$) both belong to some piece $P_1$ (resp. $P_2$). By point $(1)$, we can assume $y=x$. If $\gamma(x)=\gamma'(x)$ or $\gamma'(x)=\gamma''(x)$, we are done. Assuming that this is not the case, we will prove that $P_1=P_2$. In fact, in this case it is easily seen that $\gamma'(x)\neq \pi_{P_1}(p),\pi_{P_2}(p)$, and therefore $\gamma'|_{[0,x]}$ contains non-trivial final subsegments contained in $P_1$ and $P_2$. So, $P_1\cap P_2$ contains more than one point and $P_1=P_2$, as required.
\end{proof}

The importance of this notion is due to the following lemma:

\begin{lemma}\label{diffpattconc:lem}
If $\gamma$ and $\gamma'$ are geodesic starting from $p$ which have different initial patterns then $\gamma^{-1}\gamma'$ is a geodesic.

\end{lemma}

\begin{proof}
It is clear that $\gamma^{-1}$ and $\gamma'$ concatenate well.
\end{proof}

\begin{lemma}
If $\gamma, \gamma'$ have the same initial pattern, then the induced P-geodesics $\Gamma$ and $\Gamma'$ have the same initial P-pattern.

\end{lemma}

\begin{proof}
Let $x$ and $P$ be as in the definition of having the same initial pattern. If $p$ is contained in $P$, then $P$ contains the starting and ending point of the geodesics $\gamma|_{[0,x]}$, $\gamma'|_{[0,x]}$ and therefore they are contained in $P$. In this case $\calI_\Gamma[x/2]=\calI_\Gamma[x/2]=\emptyset$ and $h_\Gamma(J)=h_{\Gamma'}(J')=i$, where $J, J'$ are maximal closed-open intervals such that $\gamma(J),\gamma'(J')$ are contained in $P$ and $i$ is chosen in such a way that $P$ is isometric to $P_i$.
\par
If $p\notin P$, then both $\gamma$ and $\gamma'$ must pass through the projection $\gamma(y)$ of $p$ on $P$ (and $y>0$). It is easy to prove (see the proof of point $(3)$ of Lemma~\ref{Pgeod:lem}) that $\Gamma$ and $\Gamma'$ have the same pattern until $y$.
\end{proof}

Denote by $\calY(\F,p)$ the quotient of the set of geodesics starting from $p$ by the equivalence relation of having the same initial pattern. The above lemma tells us that there is a well defined map $F_{\F,p}:\calY(\F,p)\to \calW$ (recall that $\calW$ is the set of equivalence classes of P-geodesics with the same initial P-pattern).

\begin{conv}\label{F^-1:conv}
Suppose that $\G$ is a tree-graded space satisfying the same requirements as $\F$, with a fixed choice of charts, and $p\in\G$. We will set, for each $w\in\calW$, $H_{\G,p}(w)=|F^{-1}(w)|$.
\end{conv}

\subsection{Counting geodesics}
Now we will analyze asymptotic cones of relatively hyperbolic groups and minimal tree-graded structures. In each asymptotic cone $X$ of a group $G$ relatively hyperbolic with respect to its subgroups $H_1,\dots,H_n$ (infinite and with infinite index in $G$), the \emph{normalized} tree-graded structure is the set of pieces containing the following:
\begin{itemize}
\item
the subsets of $X$ induced by a left $\st$coset of some $\st H_i$ which are not real trees (notice that if it is not empty it covers $X$, and if it is empty $X$ is a real tree),
\item
if the collection $\calH$ described above is a set of pieces, the transversal trees with respect to $\calH$, and $X$ otherwise.
\end{itemize}
The pieces as in the second point will be referred to, with an abuse, as transversal trees. If $\calH$ is a set of pieces, by Proposition~\ref{valency:prop} they are homogeneous real trees of valency $2^{\aleph_0}$. On the other hand, if $\calH$ is not a set of pieces, the valency of $X$ is once again $2^{\aleph_0}$. In fact, the set $P$ induced by any $H_i$ is not a point, as each $H_i$ is infinite, and $P$ belongs to a set of pieces for $X$. Hence, being a homogeneous real tree, it contains a geodesic line. So, applying Lemma~\ref{cardpieces:lem}, we easily obtain that $X$ has valency at least $2^{\aleph_0}$, and hence exactly $2^{\aleph_0}$.
\par
Let $G$ be a group hyperbolic relative to $H_1,\dots,H_n$. Let $\calP=\{(P_i,r_i)\}_{i=0,\dots,k}$ be representatives for the isometry classes of the pieces, where $P_0$ is a homogeneous real tree with valency $2^{\aleph_0}$.
We will denote by $w_t$ the class in $\calW$ of a P-geodesic $\Gamma$ with associated almost filling of $[0,1]$ simply $\{[0,1)\}$ and such that $\Gamma(0)\in P_0$.
\par

As we will see, the P-geodesics defined below are the ones represented by actual geodesics.

\begin{defn}
A P-geodesic $\Gamma$ is admissible if for each $I_1=[p_1,q_1),I_2=[p_2,q_2)$ in its associated almost filling and such that $q_1=p_2$, $\Gamma(p_1)\notin P_0$ or $\Gamma(p_2)\notin P_0$.
\end{defn}

\par
Given a non-virtually cyclic group $G$ with cut-points in an asymptotic cone $X$, we can still consider a family $\calP=\{(P_i,r_i)\}$ of representatives for the isometry classes of the pieces in the minimal tree-graded structure, where $P_0$ is a homogeneous real tree with valency $2^{\aleph_0}$. The class $w_t$ and the definition of admissible P-geodesic make sense in this case as well.

\begin{prop}\label{F-1relhyp:prop}
If $\G$ is either
\begin{itemize}
 \item an asymptotic cone of a group $G$ hyperbolic relative to $H_1,\dots,H_n$ equipped with the normalized tree-graded structure, or
 \item an asymptotic cone of a non-virtually cyclic group $G$ containing cut-points equipped with the minimal tree-graded structure,
\end{itemize}
then for each $p\in\G$ there exists a choice of charts such that
\begin{itemize}
\item
$H_{\G,p}(w_t)=1$,
\item
$H_{\G,p}(w)=2^{\aleph_0}$ if $w_t\neq w\in\calW$ and $w$ has an admissible representative,
\item
$H_{\G,p}(w)=0$ otherwise,

\end{itemize}
where $H_{\G,p}$ is defined in Convention~\ref{F^-1:conv}.
\end{prop}

\begin{proof}

Suppose (without loss of generality) that $\G=C(G,e,\nu)$ and that $p=e$. In the case that $G$ is relatively hyperbolic, let $J\subseteq\{1,\dots,n\}$ be the set of the indices $j$ such that $C(H_j,e,\nu)$ is not a real tree.
\par
Let us consider an admissible P-geodesic $\Gamma$, with associated almost filling of $[0,l]$ $\{I_a=[p_a,q_a)\}_{a\in A}$.
\par\smallskip
 \noindent {\bf Step 1. ``Finite approximants'' for $\Gamma$.}
Consider some finite subset $A'\subseteq A$. We want to find an internal non-empty set $\calG=\calG(A')$ of internal geodesics $\hat{\gamma}$ such that their projections $\gamma$ satisfy the properties required by $\Gamma$ for $\{I_a\}_{a\in A'}$. Let us make this more precise. Choose for each $i$ an identification of $P_i$ with a piece $Q_i$ containing $e$. For each $a\in A$ choose $u_a\in\G$ such that $e$ and $u_a$ lie on $Q_{h_\Gamma(I_a)}$, where $h_\Gamma$ is as usual the index selector.
\par
We require for the internal geodesics $\hat{\gamma}\in\calG$ to satisfy the following, for each $a\in A'$
\begin{enumerate}
\item
$\hat{\gamma}(0)=e$,
\item
for each $a\in A'$ such that $h_\Gamma(I_a)\geq 1$ there exists $g=g_a\in\st G$ such that $\gamma(p_a)=[g_a]$, $\gamma(q_a)=g_a u_a$,
\item
for each $a\in A'$ such that $h_\Gamma(I_a)=0$ we have that $\gamma|_{I_a}$ is contained in a transversal tree.
\end{enumerate}
Suppose that we are able to prove that there actually exist internal geodesics with these properties for each finite $A'\subseteq A$, as we will do later. The idea is to use saturation to find an internal geodesic $\hat{\gamma}$ which satisfies $(1)-(3)$ for each $a\in A$. But first, we have ``express those properties internally'': we need to find non-empty \emph{internal} sets $\calG(A')$, to apply saturation and find that the intersection is non-empty. We will do this separately for relatively hyperbolic groups and minimal tree-graded structures, starting with the relatively hyperbolic case.
\par\smallskip
 \noindent {\bf Normalized structure case.}
Property $(1)$ requires no comments. Property $(2)$ holds if and only if there exists an infinitesimal $\rho$ such that
$$d(\hat{\gamma}(p_a\nu),g_a),d(\hat{\gamma}(p_a\nu),g_a\hat{u}_a)\leq \rho\nu,$$
where $[\hat{u}_a]=u_a$. So, fixing $\rho$, we have an internal set of internal geodesics satisfying $(2)$: the one described by the property above (we can consider a fixed $\hat{u}_a$).
\par
The task is slightly more difficult for condition $(3)$. Let $M$ be as in Lemma~\ref{M:lem}.
If the projection of $\gamma$ is contained in a transversal tree, using the property of $M$ and an argument based on Lemma~\ref{underspill:lem}, we get that there exists an infinitesimal $\eta$ such that for each left $\st$coset $H$ of some $H_j$ ($j\in J$), the diameter of $\gamma\cap N_{M}(H)$ is bounded by $\eta\nu$. We claim that the converse holds as well.
\par
In fact, consider $\gamma$ such that there are two points $\gamma(x)$ and $\gamma(y)$, with $x<y$, such that $d(\gamma(x),\gamma(y))\equiv \nu$, but there exists a left $\st$coset $H$ of some $H_j$ such that $d(x,H),d(y,H)\leq \eta'\nu$, for some infinitesimal $\eta'$. By the property of $M$, there are points $\gamma(x'),\gamma(y')$ whose distance from $H$ is at most $M$. What is more, we can assume that $d(\gamma(x'),\gamma(y'))\equiv \nu$ by taking the $x'$ as close as possible to $x$ and $y'$ as close as possible to $y$. This is a contradiction since any arc between $[x']$ and $[y']$ is contained in $P$. In particular this applies to a non-trivial subgeodesic of the geodesic $\delta$ induced by $\gamma$, and therefore $\delta$ is not contained in a transversal tree. This completes the proof of the claim.
\par\smallskip
 \noindent {\bf Minimal structure case.}
Let us move on to minimal tree-graded structures. For what regards property $(2)$, we can proceed as above. The idea for property $(3)$ is the characterization of transversal lines as those admitting a projection satisfying $(P2)$ as in the proof of Proposition~\ref{transvmin:prop}. Namely, $\gamma|_{I_a}$ is contained in a transversal tree if and only if there exists an infinitesimal $\rho$ such that $\phi(\gamma|_{[p_a\nu,q_a\nu]}, 1/\rho)$, as defined in the proof of Proposition~\ref{transvmin:prop}, holds.
\par\smallskip
 \noindent {\bf Step 2. Applying saturation.}
In both cases we have that, for each $A'$ and fixing a sufficiently large infinitesimal $\rho(A')$, we can find a non-empty (as we will show later) internal set of internal geodesics $\calG(A')$ satisfying $(1)-(3)$. Therefore, after choosing an infinitesimal greater than any $\rho(A')$ (see Lemma~\ref{cof:lem}), we can use saturation to find an internal geodesic $\hat{\gamma}$ which satisfies $(1)-(3)$ for each $a\in A$. Let $\gamma$ be the induced geodesic in $\G$. It is quite clear that we can choose identifications of each pair $(P,p)$ with some $(P_i,r_i)$, where $P$ is a piece intersecting $\gamma$ in a non-trivial subgeodesic and $p$ is the entrance point of $\gamma$ in $P$, in such a way that the P-geodesic associated to $\gamma$ is $\Gamma$. In fact, in the case when $(P,p)$ has to be identified with $(P_i,r_i)$ for $i\geq 1$, we can use $g_a$ as in property $(2)$ to ``translate'' the fixed identification of $(P_i,r_i)$ with $Q_i$. In the case when $(P,p)$ has to be identified with $(P_0,r_0)$, we can use the isotropy of $P_0$, that is the fact that for each $x,y\in P_0$ with $d(x,r_0)=d(y,r_0)$ there exists an isometry of $P_0$ fixing $r_0$ and taking $x$ to $y$. Notice that we implicitly used the fact that $\bigcup I_a$ is dense in $[0,l]$ to guarantee that the $I_a$'s are exactly the maximal intervals such that $\gamma|_{I_a}$ is contained in a piece.
\par\smallskip
 \noindent {\bf Step 3. Constructing many geodesics starting from a given one.}
So far we proved that $H_{\G,e}([\Gamma])\geq 1$ (for an admissible $\Gamma$). We want to use $\hat{\g},\gamma$ as in Step 2 to construct many other geodesics with the same properties. Again, we will start with the relatively hyperbolic case and then move on to the minimal tree-graded structure case.
\par\smallskip
 \noindent {\bf Normalized structure case.}
Consider a hyperbolic element $g\in G$. Note that the isometry induced in $\G$ by left multiplication by $g$ stabilizes no piece which is not a transversal tree. This immediately implies that, unless an initial subgeodesic of $\gamma$ is contained in the transversal tree at $e$, $\gamma$ and $g\gamma$ do not have the same initial pattern. Similarly, if $n_1\neq n_2\in\st\N$ (and $d(e,g^{n_i})\in o(\nu)$), $g^{n_1}\gamma$ and $g^{n_2}\gamma$ do not have the same initial pattern. Note that the cardinality of $\{n\in\st\N: d(e,g^{n})\in o(\nu)\}$ is at least $2^{\aleph_0}$. Therefore, if $[\Gamma]\neq w_t$ and $\Gamma$ is admissible, we have $H_{\G,e}(w)\geq2^{\aleph_0}$.
The other inequality clearly holds, so we are done in this case. On the other hand, it is clear from the definitions that $H_{\G,e}(w_t)=1$, and that no geodesic in $\G$ has non-admissible associated P-geodesic, that is $H_{\G,e}(w)=0$ if $w$ has no admissible representatives.
\par\smallskip
 \noindent {\bf Minimal structure case.}
In the minimal tree-graded structure case, we proceed approximating the argument above. Namely, consider an element $g\in\st G$ such that $[g]\neq e$ but $[g]\in T_e$. Notice that if $[\Gamma]\neq w_t$ then $d(gp,p)=2l(\gamma)+d(e,[g])$, where $p$ is the second endpoint of $\gamma$. Similarly, if $[g_1]\neq [g_2]$ and $[g_1],[g_2]\in T_e$, we have an analogous property for $g_1 p, g_2 p$. Using this, it is easily seen that for each $n\in \N^+$ we can find $g_1,\dots,g_n$ such that
\begin{enumerate}
 \item $d(g_i q,g_j q)\geq 2l(\gamma)\nu$ for $i\neq j$, where $q$ is such that $[q]=p$,
 \item $d(g_i, e)\leq 1/n$ for each $i$.
\end{enumerate}
Therefore, we can also find $g_1,\dots, g_\mu$ with the same properties for some infinite $\mu$. Unfortunately, this is not enough to conclude that $g_i \gamma$ concatenates well with $\gamma_j$ for each $i\neq j$, which would lead to the end of the proof (up to the final considerations as in the relatively hyperbolic case). However, this still holds because we can add the following requirement:
\par
$\ \ (3)$ each path in $\st G$ obtained concatenating at most $n$ internal geodesics connecting a point on $g_i \hat{\g}$ to a point in $g_j \hat{\g}$ contains a point whose distance from $g_i$ is a most $\nu/n$, if $i\neq j$.
\par
If $g_1,\dots, g_\mu$ also satisfies this property, it is easily seen that all paths connecting $g_i \gamma$ to $g_j \gamma$ for $i\neq j$ contain $e$. In particular, no piece (except transversal trees) can contain initial subpaths of both $g_i \gamma$ and $g_j \gamma$, as it does not contain cut-points. So, $g_i \gamma$ and $g_j \gamma$  concatenate well for $i\neq j$, and we are done.
\par\smallskip
 \noindent {\bf Step 4. Actually constructing the ``finite approximants''.}
We are only left with finding internal geodesics as above. For each $a\in A'$ we can find a geodesic parametrized by (a translate of the) arc length $\gamma_a:[p_a,q_a]\to\G$ which is contained in a piece isometric to $P_{h_\Gamma(a)}$, $\gamma_a(p_a)=e$ and $\gamma_a(q_a)=u_a$. Order $A'$ in such a way that $a\leq b$ if $p_a\leq p_b$. We want to show that if $a<b$, up to translating $\gamma_b$ by an element of $\st G$, we can find a geodesic parametrized by a translate of the arc length $\gamma:[p_a,q_b]\to\G$ such that $\gamma|_{[p_a,q_a]}=\gamma_a$ and $\gamma_{[p_b,q_b]}=\gamma_b$.
In fact, suppose first that $q_a<p_b$. It is easy to find a geodesic $\delta$ of length $q_b-q_a$ starting from $\gamma_a(q_a)$ such that $\gamma_a$ and $\delta$ concatenate well. Also, there exists an element $g\in\st G$ such that $g\gamma_b$ has starting point the final point of $\delta$. Up to changing $g$ we can also arrange that $\delta$ and $\gamma_b$ concatenate well (by Lemma~\ref{manypieces:lem}). The concatenation of $\gamma_a$, $\delta$ and $g\gamma_b$ is the required geodesic. If $q_a=p_b$, we can still find $g$ such that $\gamma_a$ and $\gamma_b$ concatenate well by Lemma~\ref{manypieces:lem}, unless they are both contained in a transversal tree, but this is not the case as $\Gamma$ is admissible.
\par
Using inductively the argument above (and, possibly, the first part of it for the minimum and maximum of $A'$), we obtain a geodesic $\gamma$ such that $\gamma|_{I_a}$ is contained in a piece isometric to $P_{h_\Gamma(a)}$ for each $a\in A'$. An internal geodesic connecting $e$ to an element of $\st G$ which projects on the last point of $\gamma$ satisfies all our requirements.
\end{proof}

The aim of the next subsections is to prove that the kind of information provided by the proposition above is enough to determine the bilipschitz type of $\G$.

\subsection{A criterion for being bilipschitz}
\par
Fix, throughout the subsection, a family $\{(P_i,r_i)\}_{i\in I}$ of complete homogeneous geodesic pointed metric spaces. We also assume that $P_i$ is not isometric to $P_j$ if $i\neq j$ and that no $P_i$ consists of a single point.
\begin{conv}
All tree-graded spaces from now on are assumed to be homogeneous and to have homogeneous pieces.
\end{conv}

\par
Recall that by $\calW$ we denote the set of equivalence classes of P-geodesics (with range $\{(P_i,r_i)\}$) with the same initial P-pattern, and consider a map $\alpha$ assigning to each $w\in\calW$ a cardinality $\alpha(w)$. Denote by $\F_\alpha$ a tree-graded space such that for each $p\in\F$ there exists a choice of charts such that $H_{\F_\alpha,p}(w)=\alpha(w)$ for each $w\in\calW$, if such $\F_\alpha$ exists (the question whether or not such $\F_\alpha$ exists will be addressed later).
\par
We will be interested in tree-graded spaces with pieces not necessarily isometric (but bilipschitz equivalent) to the $P_i$'s.
Suppose that we are given a family of homogeneous, geodesic complete metric space $\{Q_j\}_{j\in J}$ and bilipschitz equivalences $f_j:Q_j\to P_{i(j)}$. Suppose that $\{i(j)\}_{j\in J}=I$. Denote by $\calX$ the the set of equivalence classes of P-geodesics with range $\{(Q_j,s_j)\}$ (for $s_j\in Q_j$ such that $f_j(s_j)=r_{i(j)}$) with the same initial P-pattern. Set $\calF=\{f_j\}_{j\in J}$.
\par
Starting from the data above we are going to construct a map $\psi_\calF:\calX\to\calW$, which describes how P-geodesics change when ``substituting the $Q_j$'s with the $P_i$'s''.
\par
Indeed, we are going to construct first a map $\psi'_\calF$ from the set of P-geodesics with range $\{(Q_j, s_j)\}$ to the set of P-geodesics with range $\{(P_i,r_i)\}$, which induce the required map $\psi_\calF$.
\par
Consider a P-geodesic $\Gamma$ with range $\{(Q_j,s_j)\}$ with associated almost filling of $[0,l]$ $\{I_a=[p_a,q_a)\}_{a\in A}$.
\par
First of all, let us define the almost filling $\{J_a\}_{a\in A}$ associated to $\Theta=\psi'_\calF(\Gamma)$. We will need a function $s_\Gamma:[0,l]\to\R$, defined in such a way that it keeps track of the ``stretching factor'' given by the $f_j$'s.
\par
For $a\in A$ set $A_a=\{b\in A| q_b\leq p_a\}$. If $t\in [0,l]\backslash \bigcup_{a\in A}\mathring{ I_a}$, set
$$s_\Gamma(t)= \sum_{b\in A_a} d(r_{i(h_\Gamma(b))},f_{h_\Gamma(b)}(\Gamma(I_b)))+\lambda\left([0,p_a]\backslash\bigcup_{b\in A_a} I_b\right),$$
where $\lambda$ is the Lebesgue measure and $h_\Gamma$ is, as usual, the index selector.
If $t\in I_a$ set
$$s_\Gamma(t)= s_\Gamma(p_a)+\frac{t-p_a}{q_a-p_a}d(r_{i(h_\Gamma(a))},f_{h_\Gamma(a)}(\Gamma(I_a))).$$

\begin{rem}\label{kbilip}
It is not difficult to prove that $s_\Gamma$ is $k-$bilipschitz.
\end{rem}

We are ready to define
$$J_a=s_\Gamma(I_a).$$
Now, simply set $\Theta(J_a)=f_{h_\Gamma(a)}(\Gamma(I_a))$. It is easily shown that $\{J_a\}$ is an almost filling and that $\Theta$ is a P-geodesic whose initial pattern does not depend on the choice of the representative of $[\Gamma]\in\calX$. In particular, $\psi'_\calF$ induce a well defined map $\psi_\calF:\calX\to\calW$.

\begin{rem}
$\psi_\calF$ is surjective (because $\{i(j)\}_{j\in J}=I$).

\end{rem}

\begin{conv}
With an abuse of notation, if $\gamma$ is a geodesic (and a choice of charts has been fixed) we will denote $\psi'_\calF(\Gamma)$ and $s_\Gamma$, where $\Gamma$ is the P-geodesic associated to $\gamma$, simply by $\psi'_\calF(\gamma)$ and $s_\gamma$, respectively.
\end{conv}

\begin{thm}\label{univmakesense:thm}
Suppose that $\alpha(w)$ is infinite for each $w\in\calW$. Also, suppose that $\F$ is tree-graded with trivial transversal trees and that each of its pieces is isometric to one of the $Q_j$'s as above, and that each $f_j$ is $k-$bilipschitz. Also, suppose that for each $p\in\F$ there exists a choice of charts such that:
\begin{enumerate}
\item
$\sum_{\{x'\in\calX:\psi_\calF(x')=\psi_\calF(x)\}} H_{\F,p}(x')\leq\alpha(\psi_\calF(x))$ for each $x\in \calX$.
\medskip

Then $\F$ admits a $k-$bilipschitz embedding into $\F_\alpha$.

\medskip

\item
$\sum_{\{x'\in\calX:\psi_\calF(x')=\psi_\calF(x)\}} H_{\F,p}(x')=\alpha(\psi_\calF(x))$ for each $x\in \calX$ and $p\in\F$.

\medskip

Then $\F$ is $k-$bilipschitz equivalent to $\F_\alpha$.

\end{enumerate}

\end{thm}

\begin{rem}
Unfortunately, in the case we are interested in not all the cardinalities are infinite. However, modifying slightly the proof one can obtain Theorem~\ref{comparable:thm}. This would be a shorter way to prove that theorem than the one we will follow, that is reducing to the case when all cardinalities are infinite. We will do that to obtain an ``explicit'' description of the asymptotic cones of relatively hyperbolic group (as the universal tree-graded space described in the proof of Theorem~\ref{existence:thm}).

\end{rem}

\par

\textit{Proof of Theorem \ref{univmakesense:thm}.}\quad
We prove $(2)$, the proof of $(1)$ being very similar. Set $\G=\F_\alpha$. During the proof bilipschitz maps are implied to be $k-$bilipschitz.
\par
If $X\subseteq \F$ and $x\in X$ denote by $\calY(X,x)$ the set of elements of $\calY(\F,x)$ which can be represented by a geodesic contained in $X$. We will call a subspace $X$ of $\F$ \emph{good} if it has the following properties:
\begin{enumerate}
\item
$X$ is geodesic.
\item
For each $x\in X$, the set $\calY(X,x)$ either has at most 2 elements or it coincides with $\calY(\F,x)$. In the first case $x$ will be called empty for $X$, while in the second case it will be called full for $X$.
\item
If $X$ contains a non-trivial geodesic contained in one piece (or, equivalently, if it contains 2 points on the same piece), then it contains the entire piece.
\end{enumerate}
Analogous definitions can be given for $\G$. Note that an increasing union of good subspaces is a good subspace. Also, remark that if $X$ is a good subspace of $\F$ or $\G$ and $x,y\in X$, then \emph{any} geodesic between $x$ and $y$ is contained in $X$ (i.e., $X$ is convex). In fact, if $\gamma,\gamma'$ are geodesics connecting $x$ and $y$ and $p\in\gamma\backslash\gamma'$, there exists a piece containing $p$ and intersecting both $\gamma$ and $\gamma'$ in a non-trivial arc (as $p$ is contained in a simple loop which is a union of two subgeodesics of $\gamma$ and $\gamma'$). Therefore, conditions $(1)$ and $(3)$ imply the claim.
\par
We wish to construct the required bilipschitz equivalence using Zorn's Lemma on the set of \emph{good pairs}, that is pairs $(X,f)$ such that
\begin{itemize}
\item
$X$ is a good subspace of $\F$,
\item
$f$ is a bilipschitz embedding of $X$ into $\G$ which preserves fullness, that is $f(x)$ is full for $f(X)$ whenever $x$ is full for $X$,
\item
if, for some piece $P\subseteq\G$, $f(X)\cap P$ contains at least 2 points, there exists a piece $P'$ of $\F$ such that $f(P')=P$.
\end{itemize}
Note that if $(X,f)$ is a good pair, $f(X)$ is a good subspace of $\G$ (we require the third property in order to have this). A point is a good subspace, therefore the set such pairs is not empty.
If we set $(X,f)\leq(Y,g)$ when $X\subseteq Y$ and $g|_X=f$, then clearly any chain has an upper bound. Therefore there exists a maximal element $(M,h)$. We want to show that $M=\F$. Note that $M$ is closed, because $h$ can be extended to $\overline{M}$ as $\G$ is complete, and, as we are going to show, $\overline{M}$ is a good subspace.
\par
Let us prove that $\overline{M}$ satisfies $(3)$ first, as it is the simplest condition to check. If $[x,y]$ is a non-trivial geodesic contained in a piece $P$ and $x',y'$ are sufficiently close to $x$ and $y$ respectively, then any geodesic from $x'$ to $y'$ contains a non-trivial subgeodesic contained in $P$. This readily implies $(3)$.
\par
Let us prove $(1)$. Consider any $x,y\in\overline{M}$. We want to show that there is a geodesic contained in $\overline{M}$ which connects them. Consider any geodesic $\gamma$ in $\F$ from $x$ to $y$. Consider any piece $Q$ which intersects $\gamma$ in a non-trivial arc between $\pi_Q(x)=x'$ and $\pi_Q(y)=y'$. Each geodesic between points close enough to $x$ and $y$ intersects $Q$ in a non-trivial arc, and this readily implies, by conditions $(1)$ and $(3)$ for $M$, that $Q\subseteq M$. This argument shows that there exists a dense subset of $\gamma$ contained in $M$ (see Lemma~\ref{piecesubdense:lem}). By the remark that each geodesic connecting two points in $M$ is contained in $M$, we have that $\gamma\backslash\{x,y\}\subseteq M$, and therefore $\gamma\subseteq\overline{M}$.
\par
We are left to show $(2)$. First, we have that $\calY(M,x)=\calY(\overline{M},x)$ if $x\in M$. In fact, if $\gamma$ represents an element of $\calY(\overline{M},x)$, by the previous argument $\gamma\backslash\{y\}$ is contained in $M$, where $y$ is the last point of $\gamma$. An initial subgeodesic $\gamma'$ of $\gamma$ is contained in $M$ and so $[\gamma]=[\gamma']\in\calY(M,x)$. Also, $\calY(\overline{M},x)$ cannot contain more than one element if $x\in\overline{M}\backslash M$. Suppose in fact that this is not the case and consider geodesics $\gamma_1,\gamma_2\subseteq \overline{M}$ such that $[\gamma_1]\neq[\gamma_2]\in \calY(\overline{M},x)$. By Lemma~\ref{diffpattconc:lem}, the concatenation $\gamma$ of $\gamma_2^{-1}$ and $\gamma_1$ is a geodesic. By the proof of point $(1)$, we would have $x\in M$, as it is not an endpoint of $\gamma$.
\par
We have thus proved that $\overline{M}$ is good, so $M=\overline{M}$ by maximality and $M$ is closed.
\par
Assume that there exists $x\notin M$. Consider some $p'\in M$ and let $p$ be the last point on a geodesic $[p',x]$ which lies on $M$. We want to show that $p$ is empty for $M$, by showing that $[p,x]$ is not a representative of an element in $\calY(M,p)$. In fact, suppose  that this is not the case. Then there exists a point $q\neq p$ on $[p,x]$, a piece $P$ and a point $r\in P$, $r\neq p$, such that a geodesic $[p,r]$ is contained in $M$ and $q,r\in P$. If $p\in P$, by property $(3)$ we have $P\subseteq M$ and in particular $q\in M$, which contradicts our choice of $p$. If $p\notin P$, we can assume $r=\pi_P(p)$. So, $[p,q]$ must contain $r\in M$. This is a contradiction as $[p,q]\cap M=\{p\}$.
\par
Fix a set of representatives $R_1$ (resp. $R_2$) for the elements of $\calY(\F,p)\backslash \calY(M,p)$ (resp. $\calY(\G,p)\backslash \calY(h(M),h(p))$). At first, we want to extend $h$ to the union $M'$ of $M$ and all the elements of $R_1$.
\par
We wish to prove that up to changing representatives of $R_i$, there is a bijection $b:R_1\to R_2$ such that, for each $\gamma\in R_1$,
\begin{itemize}
\item
the P-geodesic associated to $b(\gamma)$ is $\psi'_\calF(\gamma)$ (for some choices of charts).
\end{itemize}

Consider choices of charts for $\F$ and $\G$ as in the statement.
\par
We clearly have
$$|F^{-1}_{\F,p}\left(\{x'\in\calX:\psi_\calF(x')=\psi_\calF(x)\}\right)\backslash \calY(M,p)|= \sum_{\{x'\in\calX:\psi_\calF(x')=\psi_\calF(x)\}} H_{\F,p}(x')=$$
$$H_{\G,f(p)}(\psi_\calF(x))=|F^{-1}_{\G,f(p)}(\psi_\calF(x))\backslash \calY(f(M),f(p))|$$
for each $x\in \calX$ (the first and last equality hold as each $\alpha(w)$ is infinite and $|\calY(M,p)|,|\calY(f(M),f(p))|\leq 2$ by emptiness).
\par
These considerations imply that we can choose a bijection $b:R_1\to R_2$ such that for each $\gamma\in R_1$ the P-geodesic associated to $b(\gamma)$ represents the same class in $W$ as $\psi'_\calF(\gamma)$.
To obtain what we need is now sufficient to substitute geodesics in $R_1$ and $R_2$ with appropriate subgeodesics.
\par
We are now ready define an extension of $h$, denoted by $\overline{h}:M'\to\F_\alpha$, as follows:
$$
\overline{h}(x)=
\left\{
\begin{array}{lll}
h(x) & {\rm if} & x\in M \\
b(\gamma)(s_\gamma(t)) & {\rm if} & x=\gamma(t) {\rm\ for\ some\ }\gamma\in R_1\\
\end{array}
\right.
$$
Note that $\overline{h}$ is indeed a bilipschitz embedding (see Lemma~\ref{diffpattconc:lem} and Remark~\ref{kbilip}). The last step is to extend it further so that the domain satisfies property $(3)$. Consider a piece $P$ which intersects some $\gamma\in R_1$ in a non-trivial subgeodesic $\gamma'$. As the P-geodesic associated to $b(\gamma)$ is $\psi'_\calF(\gamma)$, $\overline{h}(\gamma')$ is contained in a piece $P'$ bilipschitz equivalent to $P$.
Not only that: using the fixed choices of charts and the maps $f_i$'s, we have that there exists a $g_P:P\to P'$ which maps $\gamma'$ to $\overline{h}(\gamma')$. Let $\Delta$ be the family of pieces $P$ as above. Consider the bilipschitz equivalences $\{g_P\}_{P\in\Delta}$. We can use them to further extend $\overline{h}$ to $\widetilde{h}:M''\to\G$, where $M''=M'\cup \bigcup_{P\in\Delta} P$, as follows:
$$
\widetilde{h}(x)=
\left\{
\begin{array}{lll}
\overline{h}(x) & {\rm if} & x\in M' \\
g_P(x) & {\rm if} & x\in P \\
\end{array}
\right.
$$
Once again, this is a bilipschitz embedding. It is quite clear that $M''$ satisfies $(1)$ and $(3)$. It is also not difficult to see that it satisfies $(2)$ as well, and more precisely that
\begin{itemize}
\item
$p$ is full for $M''$,
\item
each point in $M\backslash\{p\}$ is empty (resp. full) for $M''$ if and only if it is empty (resp. full) for $M$,
\item
each point in $M''\backslash M$ is empty for $M''$.
\end{itemize}
Also, $h$ is readily checked to satisfy all the requirements needed to establish that $(M'',h)$ is a good pair. By maximality of $M$, this is a contradiction.
\par
We finally proved that if $(M,h)$ is a maximal good pair, then $M=\F$. Therefore, there exists a bilipschitz embedding $h:\F\to\G$, with the further property that $h$ preserves fullness. Let us show that this implies that $h$ is surjective. Consider, by contradiction, some $x\in\G\backslash h(\F)$. Fix some $p\in\F$ and consider a geodesic $[h(p),x]$. As $h(\F)$ is closed, being a complete metric space, we can assume that $[h(p),x]\cap h(\F)=\{h(p)\}$. Repeating an argument we already used for $M$ (recall that $h(\F)$ is a good subspace), we have that $[h(p),x]$ represents an element of $\calY(\G,h(p))\backslash \calY(h(\F),h(p))$. But $p$ is full for $\F$, so this contradicts the hypothesis that $h$ preserves fullness.
\hspace*{\fill}$\Box$

\subsection{Universal tree graded spaces}
As in the previous subsection, consider a family $\{(P_i,r_i)\}_{i\in I}$ of complete homogeneous geodesic pointed metric spaces.
\par
At this point it is a natural problem to find those maps $(w\in\calW)\mapsto \alpha(w)$ , where each $\alpha(w)$ is a cardinality, which are realized by a homogeneous tree-graded space.
\par
Our aim is now to construct a ``universal'' tree-graded space, given an infinite cardinality $\alpha(w)$ for each $w\in\calW$.

\begin{thm}\label{existence:thm}
Consider any map $(w\in\calW)\mapsto \alpha(w)$, where each $\alpha(w)$ is infinite. There exists a tree-graded space $\F=\F_\alpha$ with trivial transversal trees such that each of its pieces is isometric to one of the $P_i$'s and, for an appropriate choice of charts, for each $w\in \calW$ and $p\in\F$ we have $H_{\F,p}(w)=\alpha(w)$.
\end{thm}

We will say that $\F$ as above is a universal tree-graded space.

\begin{rem}
If we did not require the $\alpha(w)$'s to be infinite the theorem would be false, for if the cardinality of some $F_{\F,p}^{-1}(w)$ is a most one, many other cardinalities are forced to be $0$ (for the same reason why only admissible P-geodesics are represented by geodesics in the asymptotic cone of a relatively hyperbolic group). It seems reasonable that the theorem can be extended (in the same generality) to the case when the $\alpha(w)$'s are at least 2.
\end{rem}

\begin{proof}
Denote by $W$ the set of all P-geodesics. For $\Gamma\in W$, we will denote by $[\Gamma]$ the corresponding class in $\calW$.
\par
Let us define $\F$, at first as a set. Some of the definitions which follow are inspired by the definitions of $A_\mu$ and of its distance in~\cite{DP}.
Set $\alpha=\sup_{w\in\calW}\alpha(w)$. We will need to fix for each $i$ and $x\in P_i$ different from $r_i$ an isometry $\phi_x$ of $P_i$ which maps $x$ to $r_i$.
If $\Gamma$ is a P-geodesic with associated almost filling of $[0,l]$ $\calI=\{[p_a,q_a)\}$, and $x<y\in[0,l]$ do not lie in the interior of any $I\in\calI$, denote by

\begin{itemize}
\item
$-\Gamma$ the P-geodesic with associated almost filling (once again of $[0,l]$) $\{[l-q_a, l-p_a)\}$ and such that $-\Gamma(l-q_a)=\phi_{\Gamma(p_a)}(r_{h_\Gamma(p_a)})$ (where $h_\Gamma$ denotes as usual the index selector of $\Gamma$),
\item
$\Gamma^{x,y}$ the P-geodesic with associated almost filling (of $[0,y-x]$) $\{[p_a-x,q_a-x): p_a\geq x, q_a\leq y\}$ and such that $\Gamma^{x,y}(t)=\Gamma(t+x)$.
\end{itemize}

The idea is that $-\Gamma$ moves backwards along $\Gamma$, and $\Gamma^{x,y}$ is a restriction of $\Gamma$.
\par

The elements of $\F$ will be quadruples $f=(\rho_f,\Gamma_f,\calI_f,\beta_f)$ such that

\begin{enumerate}
\item
$\rho_f\in\R_{\geq 0}$,
\item
$\calI_f$ is an almost filling of $[0,\rho_f]$,
\item
$\Gamma_f$ is a P-geodesic with associated almost filling $\calI_f$,
\item
$\beta_f:[0,\rho_f)\to \alpha$ is piecewise constant from the right, that is for each $t$ there exists $\epsilon>0$ such that $f|_{[t,t+\epsilon]}$ is constant,
\item
if $x$ lies in the interior of some $I\in \calI_f$, $\beta_f$ is constant in a neighborhood of $x$,
\item
$\beta_f(t)<\alpha([\Gamma^{t,\rho_f}_f])$ for each $t\in[0,\rho_f)$ such that $t=0$ or $\beta_f$ is not constant in a neighborhood of $t$.
\end{enumerate}
Let us construct some examples of elements of $\F$. If $x\in P_i$ and $\mu<\alpha$, denote by $f^{x,\mu}$, if it exists, the element of $\F$ such that
\begin{itemize}
\item
$\rho_{f^{x,\mu}}=d_{P_i}(r_i,x)$,
\item
$\calI_{f^{x,\mu}}=\{[0,\rho_{f^{x,\mu}})\}$,
\item
$\Gamma_{f^{x,\mu}}(0)=x$,
\item
$\beta_{f^{x,\mu}}$ is constantly $\mu$,
\end{itemize}

Condition $(6)$ restricts the possible values of $\mu$.
\par
We are now going to define a concatenation of elements of $\F$. Consider $f,g\in\F$. The concatenation $f*g$ is the element of $\F$ such that

\begin{itemize}
\item
$\rho_{f*g}=\rho_f+\rho_g$,
\item
$\calI_{f*g}=\calI_f\cup \{\rho_f+I:I\in\calI_g\}$,
\item
$\Gamma^{0,\rho_f}_{f*g}=\Gamma_f$ and $\Gamma^{\rho_f,\rho_f+\rho_g}_{f* g}=\Gamma_g$
\item
$\beta_{f*g}(t)=\beta_f(t)$, where $\beta_f(t)$ is defined and $\beta_{f*g}(t)=\beta_g(t-\rho_f)$ where $\beta_g(t-\rho_f)$ is defined,
\end{itemize}

We want now to define a distance on $\F$. Consider $f,g\in \F$. Let $s=s(f,g)$ be their separation moment, i.e.
$$s=\sup\{t|\forall t'\in[0,t]\ \Gamma_f(t')=\Gamma_g(t'), \beta_f(t')=\beta_g(t') \}.$$
Note that this supremum is never a maximum. We will consider 2 cases.
\begin{itemize}
\item
$(a)$ If $\beta_f(s)=\beta_g(s)$ and $h_{\Gamma_f}(s)=h_{\Gamma_g}(s)=i$ (in particular they are defined in $s$), denoting by $J_f\in\calI_f$ and $J_g\in\calI_g$ the intervals containing $s$, we set

$$d(f,g)=(\rho_f-s)+(\rho_g-s)+d_{P_i}(\Gamma_f(s),\Gamma_g(s))-l(J_f)-l(J_g),$$
\item
$(b)$ in any other case
$$d(f,g)=(\rho_f-s)+(\rho_g-s).$$

\end{itemize}
For later purposes, define $u=u(f,g)$ and $v=v(f,g)$ in the following way:
\begin{itemize}

\item
if $d(f,g)$ is as in case $(a)$, let $u$ and $v$ be such that $J_f=[s,u)$, $J_g=[s,v)$.
\item
if $d(f,g)$ is as in case $(b)$, set $u=v=s$,
\end{itemize}

The following remark will be used many times in the proof that $d$ is a distance.

\begin{rem}\label{sepmom:rem}
\begin{tabular}{c}
\end{tabular}
\begin{itemize}
\item
$s$ does not lie in the interior of any element of $\calI_f$ or $\calI_g$,
\item
if $u>s$ or $v>s$, then both inequalities hold and $[s,u)\in\calI_f, [s,v)\in\calI_g$,
\item
$s\leq u\leq\rho_f$, $s\leq v\leq\rho_g$,
\item
the formula in case $(a)$ can be rewritten as $d(f,g)=(\rho_f-u)+(\rho_g-v)+d_{P_i}(\Gamma_f(s),\Gamma_g(s))$,
\item
$(\rho_f-u)+(\rho_g-v)\leq d(f,g) \leq (\rho_f-s)+(\rho_g-s)$,
\item
if $s(f,g)<s(g,h)$, then $s(f,h)=s(f,g)$.

\end{itemize}
\end{rem}

\begin{lemma}
$d$ is a distance.

\end{lemma}

\begin{proof}
The only non trivial property to check is the triangular inequality. Consider $f,g,h\in\F$. We have to show that $d(f,h)\leq d(f,g)+d(g,h)$. Set $s_1=s(f,g)$, $s_2=s(g,h)$ and $s_3=s(f,h)$. Define analogously $u_i$ and $v_i$, $i=1,2,3$. We will consider several cases, which cover all possible situations up to exchanging the roles of $f$ and $h$ (and therefore, for example, $u_1$ and $v_2$).
\par
$1)$ $u_1\leq s_3$, $v_2\leq s_3$. In this case we get
$$d(f,h)\leq (\rho_f-s_3)+(\rho_h-s_3)\leq (\rho_f-u_1)+(\rho_h-v_2)\leq $$
$$(\rho_f-u_1)+(\rho_g-v_1)+(\rho_g-u_2)+(\rho_h-v_2)\leq d(f,g)+d(g,h).$$
\par
$2)$ $s_3< u_1\leq u_3$, $s_1<u_1$ and $v_2\leq v_3$. We have that $[s_3,u_3)$ and $[s_1,u_1)$ both belong to $\calI_f$ and their intersection in not empty. Therefore $[s_3,u_3)=[s_1,u_1)$, that is, $s_3=s_1$ and $u_3=u_1$. Also, clearly $s_2\geq s_3$, by the definition of separation moment. We will consider 2 subcases.
\par
$2')$ $s_2<v_2$. In this case, by the same argument we just used, $s_2=s_3=s_1$ and $v_2=v_3$. For $i=h_{\Gamma_f}(s_3)$ (we will not repeat this), using the relations we found so far and the triangular inequality in $P_i$, we have that
$$d(f,h)=(\rho_f-u_3)+d_{P_i}(\Gamma_f(s_3),\Gamma_h(s_3))+(\rho_h-v_3)\leq $$
$$(\rho_f-u_1)+d_{P_i}(\Gamma_f(s_1),\Gamma_g(s_1))+(\rho_g-v_1)+(\rho_g-v_2)+d(\Gamma_g(s_2),\Gamma_h(s_2))+(\rho_h-v_2)$$
$$= d(f,g)+d(g,h).$$
\par
If $s_2=v_2$, we have $s_2\in[s_3,v_3]$. But $s_2$ cannot belong to the interior of $[s_3,v_3)\in\calI_h$. Therefore either $s_2=s_3$ or $s_2=v_3$.
But $s_2=s_3$ is contradictory as it implies $\beta_g(s_2)=\beta_g(s_1)=\beta_f(s_1)=\beta_f(s_3)=\beta_h(s_3)=\beta_h(s_2)$ and similarly $h_{\Gamma_g}(s_2)=h_{\Gamma_h}(s_2)$, therefore we should have $s_2<v_2$.
\par
$2'')$ $v_2=s_2=v_3$. As $[s_1,v_1)\in\calI_g$, $[s_1,v_2)=[s_3,v_3)\in\calI_h$ and $v_3$ is the separation moment of $g$ and $h$, we get $v_1=v_2$. We have, using $s_1=s_3<s_2$ (and the definition of separation moment),
$$d(f,h)=(\rho_f-u_3)+d_{P_i}(\Gamma_f(s_3),\Gamma_h(s_3))+(\rho_h-v_3)= $$
$$(\rho_f-u_1)+d_{P_i}(\Gamma_f(s_1),\Gamma_g(s_1))+(\rho_h-v_2)\leq $$
$$(\rho_f-u_1)+d_{P_i}(\Gamma_f(s_1),\Gamma_g(s_1))+ (\rho_g-v_1)+(\rho_h-u_2)+(\rho_h-v_2)\leq d(f,g)+d(g,h).$$
\par
$3)$ $s_3< u_1\leq u_3$, $s_1=u_1$ and $v_2\leq v_3$. As $s_1$ cannot lie in the interior of $[s_3,u_3)\in\calI_f$, $s_1=u_1=u_3$. Up to exchanging the roles of $f$ and $h$ we already treated the case when $s_2<v_2$ (case $2''$). So, we can assume $s_2=v_2$. As $s_1=u_3>s_3$, we have $s_2=s_3$, in particular $s_1>s_2$.
But $\beta_g(s_2)=\beta_f(s_2)=\beta_f(s_3)=\beta_h(s_3)=\beta_h(s_2)$ and analogously $h_{\Gamma_g}(s_2)=h_{\Gamma_h}(s_2)$, so we should have $s_2<v_2$, a contradiction. This (sub)case is therefore impossible.
\par
$4)$ $u_1=u_3=s_3$, $v_2\leq v_3$. Note that $v_3=s_3$, and so $v_2\leq s_3$.
$$d(f,h)= (\rho_f-s_3)+(\rho_h-s_3)\leq(\rho_f-u_1)+(\rho_g-v_1)+(\rho_g-u_2)+(\rho_h-v_2)\leq $$
$$d(f,g)+d(g,h).$$

\par
$5)$ $u_1>u_3>s_3$. In this case we have $s_1\geq u_3$ (if $s_1=u_1$ it is obvious, if $s_1<u_1$ it follows from the fact that $u_3$ cannot lie in the interior of $[s_1,u_1)$). Also, $s_2=\min\{s_3,s_1\}=s_3$. Observe that $s_2<u_2$, as $\beta_g(s_2)=\beta_f(s_2)=\beta_f(s_3)=\beta_h(s_3)=\beta_h(s_2)$ and similarly $h_{\Gamma_g}(s_2)=h_{\Gamma_h}(s_2)$ (we used $s_1\geq u_3>s_3=s_2$, $s_3=s_2$ and $u_3>s_3$).
Note that $v_2=v_3$ and $u_2=u_3$. In fact, $[s_2,v_3)=[s_3,v_3)\in\calI_h$ and $[s_2,u_3)=[s_3,u_3)\in\calI_f$, but also $[s_3,u_3)\in\calI_g$ as $u_3<u_1$. If $s_1=u_1$, we have
$$d(f,h)=(\rho_f-u_3)+d_{P_i}(\Gamma_f(s_3),\Gamma_h(s_3))+(\rho_h-v_3)\leq $$
$$(\rho_f-s_1)+(s_1-u_3)+2(\rho_g-s_1)+d_{P_i}(\Gamma_g(s_2),\Gamma_h(s_2))+(\rho_h-v_2)=$$
$$ (\rho_f-s_1)+(\rho_g-s_1)+(\rho_g-u_2)+d_{P_i}(\Gamma_g(s_2),\Gamma_h(s_2))+(\rho_h-v_2)=d(f,g)+d(g,h).$$
If $s_1<u_1$ and $j=h_{\Gamma_f}(s_1)$, the chain of inequalities can be modified as follows:
$$d(f,h)=(\rho_f-u_3)+d_{P_i}(\Gamma_f(s_3),\Gamma_h(s_3))+(\rho_h-v_3)\leq $$
$$(\rho_f-u_1)+(u_1-s_1)+(s_1-u_3)+2(\rho_g-v_1)+[(v_1-s_1)-(v_1-s_1)]+$$
$$d_{P_i}(\Gamma_g(s_2),\Gamma_h(s_2))+(\rho_h-v_2)= $$
$$(\rho_f-u_1)+d_{P_j}(\Gamma_f(s_1),r_j)+(\rho_g-v_1)+(\rho_g-u_2)-(v_1-s_1)+$$
$$d_{P_i}(\Gamma_g(s_2),\Gamma_h(s_2))+(\rho_h-v_2)\leq $$
$$(\rho_f-u_1)+ d_{P_j}(\Gamma_f(s_1),\Gamma_g(s_1))+(v_1-s_1)-(v_1-s_1)+(\rho_g-v_1) +d(g,h)=$$
$$d(f,g)+d(g,h).$$

\par
$6)$ $u_1>u_3=s_3$. As in case $5)$, $s_1\geq u_3$, so $s_1\geq s_3$. Note that $s_2\geq s_3$. If $s_1=u_1$, we also have $s_1>s_3$ and hence $s_2=s_3$. Also, $\beta_g(s_2)=\beta_f(s_2)=\beta_f(s_3)\neq\beta_h(s_3)=\beta_h(s_2)$, hence $u_2=v_2=s_2$ (we used $s_1>s_3=s_2$, $s_3=u_3$ and $s_3=s_2$).
$$d(f,h)=(\rho_f-s_3)+(\rho_h-s_3)\leq (\rho_f-s_1)+(s_1-s_3)+2(\rho_g-s_1)+(\rho_h-s_2)=$$
$$(\rho_f-s_1)+(\rho_g-s_1)+(\rho_g-s_2)+(\rho_h-s_2)=d(f,g)+d(g,h).$$
We are left to deal with the case $s_1<u_1$, which has 2 subcases
\par
$6')$ $s_1=s_3$. In this case $s_2=s_3$, for otherwise (i.e. for $s_2>s_3=s_1$) we would have $\beta_f(s_3)=\beta_f(s_1)=\beta_g(s_1)=\beta_h(s_1)=\beta_h(s_3)$ and similarly $h_{\Gamma_f}(s_3)= h_{\Gamma_h}(s_3)$, so $s_3<u_3$ (we used $s_3=s_1$, $s_1<u_1$, $s_1<s_2$ and $s_1=s_3$). Also, $s_2=u_2$ as $\beta_g(s_2)=\beta_g(s_1)=\beta_f(s_1)=\beta_f(s_3)\neq\beta_h(s_3)=\beta_h(s_2)$ (we used $s_2=s_1$, $s_1<u_1$, $s_1=s_3$, $s_3=u_3$ and $s_3=s_2$).
\par
$6'')$ $s_1>s_3$. Also in this case $s_2=s_3$, and $s_2=u_2$ because $\beta_g(s_2)=\beta_g(s_3)=\beta_f(s_3)\neq\beta_h(s_3)=\beta_h(s_2)$.
\par
In both cases $6')$ and $6'')$ the following estimate holds:
$$d(f,h)=(\rho_f-s_3)+(\rho_h-s_3)\leq (\rho_f-u_1)+d_{P_i}(\Gamma_f(s_1),r_i)+(s_1-s_3)+$$
$$2(\rho_g-v_1)+(\rho_h-s_2)\leq$$
$$(\rho_f-u_1)+d_{P_i}(\Gamma_f(s_1),\Gamma_g(s_1))+(\rho_g-v_1)+(v_1-s_1)+(s_1-s_3)+(\rho_g-v_1)+(\rho_h-s_2)$$
$$(\rho_f-u_1)+d_{P_i}(\Gamma_f(s_1),\Gamma_g(s_1))+(\rho_g-v_1)+(\rho_g-s_2)+(\rho_h-v_2)=d(f,g)+d(g,h).$$
\end{proof}

\begin{lemma}
$\F$ is complete.
\end{lemma}

\begin{proof}
Note that $d(f,g)\geq|\rho_f-\rho_g|$. Therefore, given a Cauchy sequence $f_n$ we have that $\rho_{f_n}\to \rho$, for some $\rho\geq 0$. If for some $t\in [0,\rho)$ the sequences $\{\Gamma_{f_n}(t)\}$, $\{\beta_{f_n}(t)\}$ (which are defined at least for $n$ large enough) is definitively constant, then define $\Gamma_f(t)=\Gamma_{f_n}(t)$, $\beta_f(t)=\beta_{f_n}(t)$ for $n$ large.
This may not happen for each $t$. However, in this case, it is easily seen that there exists $t_0<\rho$ such that
\begin{itemize}
\item
$\{\Gamma_{f_n}(t)\}$, $\{\beta_{f_n}(t)\}$ are definitively constant for $t<t_0$,
\item
$\beta_{f_n}(t)$ is definitively constant for $t\in[t_0,t)$,
\item
$\Gamma_{f_n}$, for $n$ large enough, is constant on $[t_0,\rho_{f_n})$, $h_{\Gamma_{f_n}}(t_0)$ is definitively constant (say equal to $i$) and the sequence $\{\Gamma_{f_n}(t_0)\}_{n\geq n_0}$ for $n_0$ large enough is a Cauchy sequence in $P_i$.
\end{itemize}
Using the completeness of the $P_i$'s a limit for $\{f_n\}$ is easily constructed.
\end{proof}

Let us show that $\F$ is geodesic. We will need a notion of restriction of a P-geodesic $\Gamma$ to a closed subinterval. For each $i$ and any pair of points $q,q'\in P_i$ choose a geodesic $\gamma_{q,q'}$ which connects them.
Suppose that $\Gamma$ has domain $\calI=\{I_a\}$, where $\calI$ is an almost filling of $[0,l]$. Consider some $0\leq x\leq l$. First, we define the domain of $\Gamma|_{[0,x]}$ to be
$$\calJ=\{J\cap[0,x):J\in\calI\textrm{\ and\ }J\cap[0,x)\neq\emptyset\}.$$
If $J\in\calJ$ denote by $\hat{J}$ the only interval in $\calI$ such that $J=\hat{J}\cap[0,x)$. Define $\Gamma|_{[0,x]}(J)= \gamma_{r_h,\Gamma(\hat{J})}(l(J))$, where $h=h_\Gamma(\hat{J})$.
\par
We can now define, for $f\in\F$, its $\F-$restriction $f\|_{[0,x)}$ to $[0,x)$, for $0\leq x\leq\rho_f$. We set, for $t\in [0,x)$ and in the domain of $\Gamma_f$, $\Gamma_{f\|_{[0,x)}}(t)=\Gamma|_{[0,x]}(t)$ and $\beta_{f\|_{[0,x]}}(t)=\beta_f(t)$ ($\rho_{f\|_{[0,x)}}=x$).

\par

We are finally ready to describe a geodesic between $f,g\in\F$. If $d(f,g)$ is given by the formula in case $(b)$, then $\gamma$ can be easily checked to be a geodesic parametrized by arc length between $f$ and $g$, where

\par\smallskip
\begin{tabular}{c l l}
$\gamma(t)=$ & $f\|_{[0,\rho_f-t)}$ & if $0\leq t\leq \rho_f-s$\\
 & $g\|_{[0,2s-\rho_f+t)}$ & if $\rho_f-s\leq t \leq (\rho_f-s)+(\rho_g-s)$\\
\end{tabular}
\par\smallskip

If $d(f,g)$ is given by the formula in case $(a)$, set $\delta=\gamma_{\Gamma_f(s),\Gamma_g(s)}$.

Set $i=h_{\Gamma_f}(s)$, $u=u(f,g)$, $v=v(f,g)$ and $d=d_{P_i}(\Gamma_f(s),\Gamma_g(s))$. The geodesic $\gamma$ between $f$ and $g$ is given by

\par\smallskip
\begin{tabular}{c l l}
$\gamma(t)=$ & $f\|_{[0,\rho_f-t)}$              & if $0\leq t\leq \rho_f-s-u$\\
             & $f\|_{[0,s+t_1)}*f^{\delta(t)}$   & if $\rho_f-s-u\leq t\leq \rho_f-s-u+d$\\
             & $g\|_{[0,2s+t_2+t_1-d-\rho_f+t)}$ & if $\rho_f-s-u+d\leq t \leq$\\
             &                                   & $(\rho_f-s)+(\rho_g-s)+d-u-v$\\
\end{tabular}
\par\smallskip
We will call the geodesics we just described explicit geodesics.
\par
In order to prove that $\F$ is tree-graded, we have to find a candidate set of pieces. For $i\in I$ denote by $w_i\in \calW$ the class in $\calW$ of a P-geodesic $\Gamma$ with associated almost-filling (of $[0,1]$) simply $\{[0,1)\}$ with $\Gamma(0)=x$ for some $x\in P_i$, $d(x,r_1)=1$. If $f,g\in\F$ set $f\leq g$ if their separation moment is $\rho_f$ (it actually is a partial order). Given $f\in\F$, $i\in I$ and $\beta<\alpha(w_i)$, set
$$P(f,i,\beta)=\{g\in\F: f\leq g, \beta_g(\rho_f)=\beta,{\rm \ and,\ if\ }f<g, [\Gamma_g^{\rho_f,\rho_g}]=w_i, [\rho_f,\rho_g)\in\calI_g\}.$$

Each $P(f,i,\beta)$ is easily seen to be isometric to $P_i$ (the isometry $P_i\to P(f,i,\beta)$ is given by $x\mapsto f*f^{x,\beta}$). Let $\calP$ be the set of all $P(f,i,\beta)$'s. We want to show that $\F$ is tree-graded with respect to $\calP$. We will use the characterization of tree-graded spaces given by Theorem~\ref{projgrad:thm}. More precisely, we will use the version stated in Remark~\ref{mixed:rem}.
\par
First, notice that each $P\in\calP$ is geodesic and complete, being isometric to some $P_i$. In particular, they are closed in $\F$.
\par
Also, it is readily checked that each non-trivial explicit geodesic intersects a piece in a non-trivial subgeodesic. So, each geodesic triangle whose sides are explicit geodesics which intersect each $P\in\calP$ in at most one point is trivial. Therefore, if we find a projection system for $\calP$, by Lemma~\ref{onlygeod:lem} we are done.
\par
Consider $P=P(f,i,\beta)\in\calP$. For each $r\in\F$ define $\pi_P(r)$ to be the first point on the explicit geodesic between $r$ and $f$. It is obvious that $(P'1)$ holds.
\par
The following claim can be checked directly.
\begin{claim}
Suppose that $\pi_P(r_1)\neq \pi_P(r_2)$. Then the explicit geodesic from $r_1$ to $r_2$ is obtained concatenating the explicit geodesics from $r_1$ to $\pi_P(r_1)$, from $\pi_P(r_1)$ to $\pi_P(r_2)$ and from $\pi_P(r_2)$ to $r_2$.
\end{claim}

In particular, $d(r_1,r_2)=d(r_1,\pi_P(r_1))+d(\pi_P(r_1),\pi_P(r_2))+d(\pi_P(r_2),r_2)$, that is, $(P'2)$.
\par
To conclude the proof that $\F$ is tree-graded, we are left to show $(T_1)$. Consider $P(f,i,\beta)$ and $P(g,j,\delta)$. First of all $P(f,i,\beta)$ can have a point in common with $P(g,j,\delta)$ only if $f\leq g$ or vice versa.
Let us consider the case $f<g$ (the case $g<f$ is of course analogous, so we will be left to deal only with the case $f=g$).
If $h\in P(f,i,\beta)\cap P(g,j,\delta)$ (in particular $h\geq g>f$), then $\calI_h$ contains $[\rho_f,\rho_h)$. If we also had $h>g$, $\calI_h$ would contain $[\rho_g,\rho_h)$, which is different from $[\rho_f,\rho_h)$, but their intersection is not empty, a contradiction. This readily implies that if $f\in P(f,i,\beta)\cap P(g,j,\delta)$, then $\rho_h=\rho_g$, and so we must have $h=g$.
\par
In the case $f=g$, it is clear that if $i\neq j$ or $\beta\neq\delta$ then $f=g$ is the only point in $P(f,i,\beta)\cap P(g,j,\delta)$.
\par
In order to prove the theorem, we are left to show that for each $w\in \calW$, $F_{\F,p}^{-1}(w)$ has cardinality $\alpha(w)$, for some choice of charts. Choose the identification $(P_i,r_i)\to(P(f,i,\beta),f)$ to be $x\mapsto f*f^{x,\beta}$. Recall that we fixed for each $i$ and $x\in P_i$ different from $r_i$ an isometry $\phi_x$ of $P_i$ which maps $x$ to $r_i$. These isometries, together with the already fixed identifications, yield a choice of charts, which is the one we will use.
\par
Note that each equivalence class in $\calY(\F,p)$ has a representative which is an explicit geodesic, by the fact that there is an explicit geodesic connecting each pair of points in $\F$ (clearly, geodesics with the same endpoints have the same initial pattern). Therefore, in what follows we are allowed to restrict to considering explicit geodesics only.
\par
Consider any $f\in\F$. There can be 4 kinds of explicit geodesics starting from $f$, which are listed below.

\begin{enumerate}
\item
Explicit geodesics $\gamma$ such that, for each $\epsilon>0$ in the domain of $\gamma$, $f<\gamma(\epsilon)$ and $\beta_{\gamma(\epsilon)}$ is not constant in a neighborhood of $\rho_f$. In this case $F_{\F,f}([\gamma])=[\Gamma^{\rho_f,\rho_f+\epsilon}_{\gamma(\epsilon)}]$ for $\epsilon$ as above.
\item
Explicit geodesics as in point $(1)$ except that $\beta_{\gamma(\epsilon)}$ is constant in a neighborhood of $\rho_f$. There is one element in each $F_{\F,f}^{-1}(w)$ which can be represented by this kind of explicit geodesics.
\item
Explicit geodesics $\gamma$ such that $F_{\F,f}([\gamma])=[-\Gamma_f]$.
\item
Other explicit geodesics: in this case there exists an interval in $\calI_f$ of the kind $[t,\rho_f)$ such that $F_{\F,f}([\gamma])=[-\Gamma_f]$, for any $x\in P_i$, $x\neq r_i$, where $i=h_{\Gamma_f}(t)$.
\end{enumerate}

Let $G$ be the set of equivalence classes in $\calY(\F,f)$ of explicit geodesics of type $(1)$. We claim that for each $w\in\calW$, the map $H_w:(F_{\F,f}^{-1}(w)\cap G)\mapsto \alpha(w)$ given by $[\gamma]\mapsto \beta_{\gamma(\epsilon)}(\rho_f)$ (for any $\epsilon>0$ in the domain of $\gamma$) is injective and the image differs from $\alpha(w)$ for at most one element.

If this holds, as geodesics of type $(2)-(4)$ accounts for finitely many elements in each $F_{\F,f}^{-1}(w)$ and each $\alpha(w)$ is infinite, we are done.
\par
We are left to prove the claim. Let us prove ``almost-surjectivity'' first. Suppose $w=[\Gamma]$ for some P-geodesic $\Gamma$ with domain the almost filling of $[0,l]$ $\calI$. For each $\kappa<\alpha(w)$, there exists an element $g(\kappa)$ of $\F$ such that
\begin{itemize}
\item
$\rho_{g(\kappa)}=l$,
\item
$\calI_{g(\kappa)}=\calI$,
\item
$\Gamma_{g(\kappa)}=\Gamma$,
\item
$\beta_{g(\kappa)}$ is constantly $k$.

\end{itemize}

We have that the explicit geodesic $\gamma$ from $f$ to $f*g(\kappa)$ is of type $(1)$ for each but at most 1 value of $k$. As clearly $\gamma$ is contained in $F_{\F,f}^{-1}(w)$ and $H_w(\gamma)=\kappa$, ``almost-surjectivity'' is proved.
\par
For what regards injectivity, if $H_w(\gamma_1)=H_w(\gamma_2)$, by the fact that the function $\beta_*$'s are piecewise constant from the right there exists $\epsilon>0$ such that $\beta_{\gamma_1(\epsilon)}=\beta_{\gamma_2(\epsilon)}$. It is easily seen that $\gamma_1$ and $\gamma_2$ have the same pattern until $\epsilon$.
\end{proof}

\subsection{A tree-graded structure on the homogeneous real tree of valency $2^{\aleph_0}$}

Consider a homogeneous real tree $T$ of valency $2^{\aleph_0}$, and fix a base point $p\in T$. Let $\calW$ be the set of the equivalence classes of the P-geodesics with range $\{(T,p)\}$ and set $\alpha(w)=2^{\aleph_0}$ for each $w\in\calW$. A direct application of Theorem~\ref{existence:thm} shows that a universal tree-graded space $\F=\F_\alpha$ exists.

\begin{prop}
$\F$ is isometric to $T$.
\end{prop}

\begin{proof}
Being tree-graded with respect to real trees, $\F$ is a real tree. We only need to determine the valency at each $q\in\F$.
\par
It is not difficult to show that $|\calW|=2^{\aleph_0}$ (this is also the cardinality of the set of the P-geodesics with range $\{(T,p)\}$).
So, considering the partition of $\calY(\F,q)$ given by $\{F^{-1}_{\F,q}(w)\}_{w\in\calW}$, we get $|\calY(\F,q)|=2^{\aleph_0}\times 2^{\aleph_0}$.
Consider two geodesics $\gamma_1, \gamma_2$ starting at $q$, and suppose that $\gamma_1\cap\gamma_2=\{q\}$. It is clear that either $\gamma_1$ and $\gamma_2$ represent different elements of $\calY(\F,q)$ or they belong to the same piece. In particular we have that the valency of $\F$ at $q$ is at most $2^{\aleph_0}\times 2^{\aleph_0}=2^{\aleph_0}$, and so that it is exactly $2^{\aleph_0}$.
\par
Therefore the valency of $\F$ at each $q\in\F$ is $2^{\aleph_0}$. In particular, $\F$ is isometric to $T$.
\end{proof}

We are finally ready to prove Theorem~\ref{comparable:thm} and Theorem~\ref{cutpinthecone:thm}. The proof is the same for both results.

\emph{Proof of Theorem~\ref{comparable:thm} and Theorem~\ref{cutpinthecone:thm}.}
The proposition above, together with Proposition~\ref{F-1relhyp:prop} implies that $C(G_0,\nu_0)$ and $C(G_1,\nu_1)$ can be given a tree-graded structure such that for each $p\in C(G_0,\nu_0), q\in C(G_1,\nu_1)$ (there exists a choice of charts such that) $H_{C(G_0,\nu_0),p}$ and $H_{C(G_1,\nu_1),q}$ are constant functions with value $2^{\aleph_0}$.
\par
Consider now a family of homogeneous geodesic representatives $\{P_i\}$ for the classes of bilipschitz equivalence of the pieces of $C(G_0,\nu_0)$ and $C(G_1,\nu_1)$. Let $\calX$ be the set of the equivalence classes of the P-geodesics with range $\{(P_i,r_i)\}$, for some choice of $r_i\in P_i$, and set $\beta(x)=2^{\aleph_0}$ for each $x\in\calX$. Applying Theorem~\ref{existence:thm} we have that a universal tree-graded space $\G=\F_\beta$ exists.
\par
An easy application of Theorem~\ref{univmakesense:thm} gives that both $C(G_0,\nu_0)$ and $C(G_1,\nu_1)$ are bilipschitz equivalent to $\G$, and therefore they are bilipschitz equivalent.
\hspace*{\fill}$\Box$

\begin{rem}\label{ultrafilter:rem}
Notice that in the proof above we only used two facts about $C(G_0,\nu_0)$ and $C(G_1,\nu_1)$, that is that we can apply Proposition~\ref{F-1relhyp:prop} to them and that the sets of the bilipschitz equivalence classes of their pieces coincide. In particular, the proof works also if, in the definition of being comparable, we allow the asymptotic cones of $G_0$ and those of $G_1$ to be constructed using different ultrafilters.
\end{rem}

\begin{rem}\label{expldescr:rem}
In view of the tree-graded structure constructed in the previous proof, an asymptotic cone of a relatively hyperbolic group is (isometric to a) universal tree-graded space. In particular, the proof of Theorem~\ref{existence:thm} provides an ``explicit'' description of such asymptotic cones when the pieces are known.

\end{rem}


\begin{thebibliography}{1}
\bibitem{BKMM} J. Behrstock, B. Kleiner, Y. Minsky, L. Mosher - \emph{Geometry and rigidity of mapping class groups}, arXiv:0801.2006
\bibitem{Bowditch:RelHyp} B.~Bowditch, \emph{Relatively hyperbolic groups}, Preprint, University of Southampton, \texttt{http://www.maths.soton.ac.uk/pure/preprints.phtml}, 1997.
\bibitem{Dahmani:thesis}
F.~Dahmani, \emph{{Les groupes relativement hyperboliques et leurs
bords}}, Ph.D. thesis, University of Strasbourg, 2003.
\bibitem{DP} A. Dyubina, I. Polterovich - \emph{Explicit constructions of universal $\R-$trees and asymptotic geometry of hyperbolic spaces},  Bull. London Math. Soc. 33 (2001), 727--734.
\bibitem{Dr1} C. Dru\c{t}u - \emph{Quasi-isometry invariants and asymptotic cones}, Int. J. Alg. Comp. 12 (2002), 99--135.
\bibitem{Dr2} C. Dru\c{t}u - \emph{Relatively hyperbolic groups: geometry and quasi-isometric invariance}, Comment. Math. Helv. 84 (2009), 503--546.
\bibitem{DS1} C. Dru\c{t}u, M. Sapir - \emph{Tree-graded spaces and asymptotic cones of groups}, Topology 44 (2005), 959--1058.
\bibitem{Fa} B. Farb - \emph{Relatively hyperbolic groups}, Geom. Funct. Anal. 8 (1998), 810--840.
\bibitem{FLS} R. Frigerio, J.F. Lafont, A. Sisto - \emph{Rigidity of high dimensional graph manifolds}, arXiv:1107.2019v1 (2011).
\bibitem{Go} R. Goldblatt - \emph{Lectures on the hyperreals: an introduction to nonstandard analysis}, Graduate Texts in Mathematics, 188, Springer-Verlag, 1998.
\bibitem{Gr1} M. Gromov - \emph{Groups of polynomial growth and expanding maps},  Inst. Hautes \'Etudes Sci. Publ. Math. No. 53 (1981), 53--73.
\bibitem{Gr2} M. Gromov - \emph{Hyperbolic groups. Essays in group theory}, 75-263, Math. Sci. Res. Inst. Publ., 8, Springer, New York, 1987.
\bibitem{KaL1} M. Kapovich, B. Leeb - \emph{On asymptotic cones and quasi-isometry classes of fundamental groups of nonpositively curved manifolds}, Geom. Funct. Analysis no. 3 (1995), 582--603.
\bibitem{KaL2} M. Kapovich, B. Leeb - \emph{Quasi-isometries preserve the geometric decomposition of Haken manifolds}, Invent. Math. 128 no. 2 (1997), 393--416.
\bibitem{KlL} B. Kleiner, B. Leeb - \emph{Rigidity of quasi-isometries for symmetric spaces and Euclidean buildings}, Publ. Math. IHES 86 (1997), 115--197.
\bibitem{OS} D. V. Osin, M. Sapir - \emph{Universal tree-graded spaces and asymptotic cones}, to appear in Int. J. Alg. Comp
\bibitem{Os1} D. V. Osin - \emph{Relatively hyperbolic groups: Intrinsic geometry, algebraic properties, and algorithmic problems}, Mem. Amer. Math. Soc. 179 (2006).
\bibitem{Os2} D. V. Osin - \emph{Elementary subgroups of relatively hyperbolic groups and bounded generation}, Internat. J. Algebra Comput. 16 (2006), 99--118.
\bibitem{Th} A. Sisto - \emph{Tree-graded spaces and relatively hyperbolic groups}, Master's Thesis.
\bibitem{Si} A. Sisto - \emph{Separable and tree-like asymptotic cones of groups}, arXiv:1010.1199v2 (2010).
\bibitem{Si3} A. Sisto - \emph{$3$-manifold groups have unique asymptotic cones}, arXiv:1109.4674v1 (2011).
\bibitem{vDW} L. van den Dries, A. J. Wilkie - \emph{Gromov's Theorem on groups of polynomial growth and elementary logic}, J. of Algebra 89 (1984), 349--374.
\bibitem{Yaman:RelHyp} A. Yaman, \emph{A topological characterisation of relatively hyperbolic groups}, J. Reine Angew. Math. (Crelle's Journal) 566 (2004), 41--89.

\end{thebibliography}
\end{document}